\documentclass[12pt]{article}
\usepackage{fullpage, enumerate,amsthm}
\usepackage{hyperref,verbatim}
\usepackage{latexsym}
\usepackage{amsmath,amssymb}
\usepackage{enumerate}
\usepackage{caption,subcaption,xcolor}
\usepackage{tikz,graphicx, graphics,psfrag}
\usetikzlibrary{shapes}
\newtheorem{thm}{Theorem}[section]
\newtheorem{lem}[thm]{Lemma}

\newtheorem{cor}[thm]{Corollary}
\newtheorem{clm}[thm]{Claim}
\newtheorem{rem}[thm]{Remark}
\newtheorem{prop}[thm]{Proposition}
\newtheorem*{keylem}{Key Lemma}
\newtheorem*{mainthm}{Main Theorem}
\renewcommand\S{\mathcal{S}}
\newcommand\G{\mathcal{G}}

\newcommand\Zed{\mathbb{Z}}
\newcommand\card[1]{|#1|}
\renewcommand\geq{\geqslant}
\renewcommand\leq{\leqslant}
\renewcommand\deg{d}
\newcommand\AT{{\textrm{AT}}}
\newcommand\chil{\chi_\ell}

\renewcommand{\le}{\leqslant}
\renewcommand{\ge}{\geqslant}
\def\vph{\varphi}

\newcommand{\aside}[1]{\marginnote{\scriptsize{#1}}[0cm]}
\newcommand{\aaside}[2]{\marginnote{\scriptsize{#1}}[#2]}
\newcommand\Emph[1]{\textit{#1}\aside{#1}}
\newcommand\EmphE[2]{\textit{#1}\aaside{#1}{#2}}
\usepackage{marginnote}
\newenvironment{clmproof}[1]{\par\noindent\underline{Proof.}\space#1}{\leavevmode\unskip\penalty9999\hbox{}\nobreak\hfill\quad\hbox{$\diamondsuit$}\smallskip}

\def\aftermath{\par\vspace{-\belowdisplayskip}\vspace{-\parskip}\vspace{-\baselineskip}}
\title{Degeneracy and Colorings of Squares\\ of Planar Graphs without 4-Cycles}
\author{Ilkyoo Choi\thanks{Department of Mathematics,  College of Natural
Sciences, Hankuk University of Foreign Studies (HUFS), Republic of Korea; 
\texttt{ilkyoo@hufs.ac.kr}; 
Supported by the Basic Science Research Program through the National Research
Foundation of Korea (NRF) funded by the Ministry of Education
(NRF-2018R1D1A1B07043049), and also by the Hankuk University of Foreign Studies Research Fund.
} \and 
Daniel W. Cranston\thanks{Department of Mathematics and Applied Mathematics,
Virginia Commonwealth University, Richmond, VA, USA; 
\texttt{dcranston@vcu.edu}; This research is partially supported by NSA Grant
H98230-15-1-0013.} \and Th\'{e}o Pierron\thanks{%
Univ.~Bordeaux, Bordeaux INP, CNRS, LaBRI, UMR 5800, F-33400 Talence, France;
\texttt{tpierron@labri.fr}} 
}

\begin{document}
\maketitle
\abstract{
We prove several results on coloring squares of planar graphs without 4-cycles.
First, we show that if $G$ is such a graph, then $G^2$ is
$(\Delta(G)+72)$-degenerate.  This implies an upper bound of $\Delta(G)+73$ on
the chromatic number of $G^2$ as well as on several variants of the chromatic
number such as the list-chromatic number, paint number, Alon--Tarsi number, and
correspondence chromatic number.  We also show that if $\Delta(G)$ is
sufficiently large, then the upper bounds on each of these parameters of $G^2$
can all be lowered to $\Delta(G)+2$ (which is best possible).  To complement
these results, we show that 4-cycles are unique in having this property. 
Specifically, let $S$ be a finite
list of positive integers, with $4\notin S$.  For each constant $C$, we
construct a planar graph $G_{S,C}$ with no cycle with length in $S$,
but for which $\chi(G_{S,C}^2) > \Delta(G_{S,C})+C$.
}

\section{Introduction}

The \Emph{square}, $G^2$, of a graph $G$ is formed from $G$ by adding an edge
$vw$ for each pair of vertices, $v$ and $w$, at distance two in $G$.  It is easy
to check that $\chi(G^2)\le \Delta(G^2)+1 \le \Delta(G)^2+1$, and this bound
can be tight, as when $G$ is the 5-cycle or the Petersen graph (here $\chi$ and
$\Delta$ denote, respectively, the chromatic number and maximum
degree).\footnote{For simplicity, in this introduction we discuss only standard
vertex coloring.  But starting in Section~\ref{sec2} we consider degeneracy, 
and at the end of that section we mention multiple other graph coloring parameters.}  
Even when $\Delta(G)$ is arbitrarily large, there exist constructions showing
that this upper bound on $\chi(G)$ cannot be improved much.  For example, when
$G$ is the incidence graph of a projective plane,\footnote{This incidence graph $G$ is
$(k+1)$-regular and bipartite with each part of size $k^2+k+1$.  Since each pair
of vertices within a part has a common neighbor, $\omega(G^2)=k^2+k+1$.}  
we have $\chi(G^2) \approx \Delta(G)^2-\Delta(G)$.
However, for planar graphs, we have much better bounds on $\chi(G^2)$.

Recall that Euler's formula implies that every planar graph $G$ is
5-degenerate.  Coloring vertices greedily in the reverse of this
degeneracy order~\cite{Jonas},\cite[Theorem 4.9]{Guide} shows that
$\chi(G^2)\le 9\Delta(G)$.  Refinements of this approach have led to
successive improvements of this upper bound, culminating with the
result of Molloy and Salavatipour~\cite{MolloySalavatipour} that
$\chi(G^2)\le \left\lceil{5\over 3}\Delta(G)\right\rceil+78$.
Havet et al.~\cite{HvdHMR} also proved that $\chi(G^2)\le
\frac32\Delta(G)(1+o(1))$,
which strengthens the bound of~\cite{MolloySalavatipour} when $\Delta(G)$ is
sufficiently large.  Amini et al.~\cite{AEvdH} proved the same bound for
all graphs embeddable in any fixed surface.

Every graph $G$ satisfies $\chi(G^2)\ge
\Delta(G)+1$, and for planar graphs we might naively hope to prove a matching
upper bound, or at least a bound of the form $\chi(G^2)\le \Delta(G)+C$, for
some constant $C$.  However, for each $k\in \Zed^+$, Wegner constructed a planar
graph $G_k$ with $\Delta(G_k)=k$ and $\chi(G^2_k) = \left\lfloor{3\over
2}k\right\rfloor+1$; Figure~\ref{fig:wegner} shows his construction.  So to
prove a bound of the form $\chi(G^2)\le \Delta(G)+C$, we must restrict to some
proper subset of planar graphs.

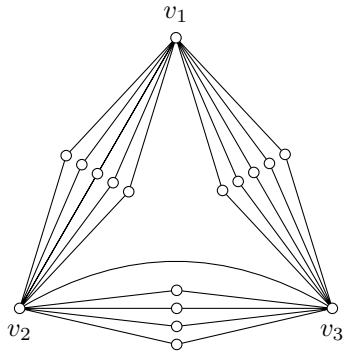
\begin{figure}[h!]
\centering
\begin{tikzpicture}
[
high/.style={inner sep=1.4pt, outer sep=0pt, circle, draw,fill=white} 
]

\def\e{3}
\def\es{1.2}
\def\et{0.65}
\def\R{2.1}
\def\r{1.5}

\begin{scope}[xshift=-2cm, yshift = 0cm, scale=0.8]  
\draw 

(90:\e) node[high,label=above:{\footnotesize $v_1$}] (v1){}
(90+120:\e) node[high,label=below:{\footnotesize $v_2$}] (v2){}
(90+240:\e) node[high,label=below:{\footnotesize $v_3$}] (v3){}

(v2)--(v1)

\foreach \i in {0,1,-1,2,-2}{
(v2) ++ (60:\e/1.16)  ++ (60+90:0.3*\i) node[high](x){}
(v2)--(x)-- (v1)
}

\foreach \i in {0,1,-1,2,-2}{
(v1) ++ (-60:\e/1.16)  ++ (-60+90:0.3*\i) node[high](x){}
(v1)--(x)-- (v3)
}

\foreach \i in {0,1,-1,2}{
(v3) ++ (-180:\e/1.16)  ++ (-180+90:0.3*\i) node[high](x){}
(v2)--(x)-- (v3)
}
;

\draw (v2) to [bend left=30] (v3);

\end{scope}

\end{tikzpicture}

{\caption{Wegner's construction \label{fig:wegner}}}
\end{figure}

Wang and Lih~\cite{WangLih} conjectured that, for each $g\ge 5$, there exists
$D_g$ such that
if $G$ is a planar graph with girth at least $g$ and $\Delta(G)\ge D_g$, then
$\chi(G^2)=\Delta(G)+1$.  This is true for $g\ge 7$~\cite{BGINT}. But it is false
for girth 5 and 6 since, for each $k\geq 3$, there exists a
planar graph $G_k$ with $\Delta(G_k)=k$ and with girth 6 such that
$\chi(G_k^2)=\Delta(G_k)+2$~\cite{BGINT}.
However, Dvo\v{r}\'ak et al.~\cite{DKNS} proved a surprising complementary
result: $\chi(G^2)\le \Delta(G)+2$ whenever $G$ is a planar graph with girth 6
and $\Delta(G)$ sufficiently large.  This work inspired analogous results
for planar graphs with (i) girth 5~\cite{BCP} and (ii) no 4-cycles or 5-cycles
(though 3-cycles are allowed)~\cite{DongXu}\footnote{Here we only hit the
highlights.  For a more detailed history of this problem, we recommend the
introduction of~\cite{HvdHMR} and~\cite[Conjecture 4.7 ff.]{Guide}.}.  In each
case the bound $\chi(G^2)\le \Delta(G)+2$
still holds (though the required lower bound on $\Delta(G)$ is larger).

The work above naturally leads to the following question.  Exactly which cycle
lengths can be forbidden from planar graphs to get a bound of the form
$\chi(G^2)\le \Delta(G)+C$?  For a set $\S$ of positive integers, let $\G_\S$
denote the family of planar graphs having no cycles with length in $\S$.

\begin{mainthm}
\label{mainthm1}
For a finite set $\S$ there exists a constant $C_\S$ such that $\chi(G^2)\le
\Delta(G)+C_\S$ for all $G\in \G_\S$ if and only if $4\in \S$.
\end{mainthm}

We prove the Main Theorem in two parts.  Immediately below we give a
construction that proves the ``only if'' part.  In
Section~\ref{sec:no-4-cycles} we handle the ``if'' part, the case when $4\in
\S$.  In fact, we prove the stronger statement that the vertices of every graph
$G\in \G_{\{4\}}$ can be ordered so that each vertex is preceded in the order
by at most $\Delta(G)+72$ of its neighbors in $G^2$.  Now the coloring result
follows by coloring greedily.  
In Section~\ref{sec3}, when $\Delta(G)$ is sufficiently large we
strengthen our bound to $\chi(G^2)\le \Delta(G)+2$, which is sharp.  This bound
also holds for paint number, Alon--Tarsi number, and correspondence chromatic
number (all defined at the end of Section~\ref{sec:no-4-cycles}).

\begin{lem}
If $4,2k\notin \S$, for some odd integer $k\ge 3$, then there does not exist a
constant $C_\S$ such that $\chi(G^2)\le \Delta(G)+C_\S$ for every $G\in \G_\S$.
\end{lem}

\begin{proof}
Begin with a $k$-cycle and replace each edge $vw$ with a copy of $K_{2,t}$, so
that the two vertices of degree $t$ replace $v$ and $w$.  The resulting graph,
$G_{k,t}$ has maximum degree $2t$ and has cycles only of lengths $4$ and $2k$.
In every proper coloring of $G^2_{k,t}$, each color class contains at most
$(k-1)/2$ vertices of degree 2 in $G_{k,t}$ (by the Pigeonhole Principle). 
Since $G_{k,t}$ has $kt$ vertices of degree 2, we get $\chi(G^2_{k,t})\ge
kt/((k-1)/2) = 2kt/(k-1) = 2t+2t/(k-1) = \Delta(G)+2t/(k-1)$.  Given any
constant $C$, we can choose $t$ sufficiently large so that $2t/(k-1) > C$.
\end{proof}

\section{Graphs with no 4-cycles}
\label{sec2}
\label{sec:no-4-cycles}
Our goal in this section is to prove Theorem~\ref{4-cycle-thm}, below.
First we need a few definitions.  A \Emph{$k$-vertex}
(resp.~$k^+$-vertex, $k^-$-vertex) is a vertex of degree equal to
(resp.~at least, at most) $k$; a \Emph{$k$-neighbor}, of a
vertex $v$, is an adjacent $k$-vertex.  Analogously, we define
\EmphE{$k$-face}{.3cm}, \emph{$k^+$-face}, and \emph{$k^-$-face}.  We
write $d(v)$ for the degree of a vertex $v$ and $\ell(f)$ for the
length of a face $f$. We write $N[v]$\aside{$N[v], N[S], N^2(v)$} to denote
$N(v)\cup\{v\}$ and $N[S]$ for $\cup_{v\in S} N[v]$.
We write $N^2(v)$ for the set of neighbors of
$v$ in $G^2$.  When the context could be unclear, we specify our meaning by
using $d_G$, $N_G$, and $N_G^2$. An order, $\sigma$, of
$V(G)$ is \Emph{good for $G$} if each vertex, $v$, of $G$ is preceded
in $\sigma$ by at most $\Delta(G)+72$ vertices in $N^2(v)$.  Following
the approach of~\cite{CJaeger}, we prove the degeneracy result below,
which immediately implies the desired coloring bounds, by coloring
greedily.

\begin{thm}
\label{4-cycle-thm}
For every planar graph $G$ with no 4-cycles, there exists a vertex
order $\sigma$ such that each vertex $v$ is preceded in $\sigma$ by at most
$\Delta(G)+72$ of its neighbors in $G^2$.
\end{thm}

Our proof of Theorem~\ref{4-cycle-thm} is by discharging, with initial charge
$d(v)-4$ for each vertex $v$ and $\ell(f)-4$ for each face $f$.  In the next
section we discuss the discharging rules, but for now it is enough to note that
we only need to give extra charge to 2-vertices, 3-vertices, and 3-faces.  
Here we prove that certain configurations are reducible; that is, they cannot
appear in a minimal counterexample.  In each case we assume that our minimal
counterexample $G$ contains such a configuration.  We modify $G$ to get a
smaller graph $G'$ (that is also planar and without 4-cycles), and which
therefore has the desired vertex order, $\sigma'$.
Finally, we modify $\sigma'$ to get $\sigma$, a good vertex order for $G$ of
$V(G)$.  Each reducible configuration formalizes the intuition that every
2-vertex, 3-vertex, and 3-face of $G$ must be near a vertex $v$ of high degree.
 This is useful, since $v$ has extra charge to share with nearby vertices and
faces that need it.
\bigskip
  
\noindent
\textit{Proof of Theorem~\ref{4-cycle-thm}.}
Suppose the theorem is false, and let $G$ be a counterexample that minimizes 
the number of $3^+$-vertices and, subject to that, the number of edges.
If $G$ is disconnected, then we can get a good vertex order for each component
(by minimality) and concatenate these to get a good order for $G$.  Thus, $G$ is
connected.  Similarly, if $G$ has a 1-vertex $v$, then $G-v$ has a good order
$\sigma'$ and we can append $v$ to $\sigma'$.  So $G$ has no 1-vertex.
A vertex $v$ is \emph{big} if $d(v)\ge 10$,  
and $v$ is \emph{small}\aside{big, small} if $5\le d(v)\le 9$.  
Note that $\Delta(G)\ge 10$, since otherwise each vertex has at most $9^2$
neighbors in $G^2$, so every vertex order shows that $G$ is not a
counterexample.

\subsection{Reducible Configurations}
\label{sec:reducibles}

\begin{keylem}\label{keylem}
For an edge $vw$ in $G$, if both $v$ and $w$ are not big, then at least one of $v$ and $w$ has at least two big neighbors. 
\end{keylem}
\begin{proof}
Suppose to the contrary that both $v$ and $w$ are not big, and that each has at
most one big neighbor.  By minimality, $G-vw$ has a good order, $\sigma'$.  By
deleting $v$ and $w$ from $\sigma'$, we get a good order (for $G$) of
$V(G)-\{v,w\}$.  Since $v$ is not big and has at most one big neighbor, 
$\card{N^2(v)}\le \Delta(G)+(10-1)(10-2)$.  By symmetry, 
$\card{N^2(w)}\leq\Delta(G)+(10-1)(10-2)$.  Thus, by appending $v$ and $w$ to
the order, we get a good order for $G$, which is a contradiction. 
\end{proof}

\begin{lem}
\label{rc-3face-2}
If a $3$-face $f$ is incident with a $2$-vertex, then the other two vertices on $f$ must be big vertices.
\end{lem}
\begin{proof}
Let $vw_1w_2$ be a $3$-face that is incident with a $2$-vertex $v$. 
Suppose to the contrary that $w_1$ is not big. 
By minimality, $G-v$ has a good order, which is a good order for $G$ of
$V(G)-\{v\}$. 
Since $w_1$ is not big, $\card{N^2(v)}\leq \Delta(G)+7$.
Thus, we can append $v$ to obtain a good order of $G$, which is a contradiction. 
\end{proof}

\begin{lem}
\label{rc-3face-33}
Every $3$-face that is incident with two $3$-vertices is also incident with a big
vertex. 
\end{lem}
\begin{proof}
Suppose that a $3$-face is incident with two $3$-vertices $v_1, v_2$ and a
vertex $w$.  Applying the Key Lemma to $v_1v_2$ shows that $w$ must be big.
\end{proof}

\begin{lem}
\label{rc-3-sss}
Every $3$-vertex has a big neighbor. 
\end{lem}
\begin{proof}
Let $v$ be a $3$-vertex with neighbors $w_1, w_2, w_3$.
Suppose to the contrary that every $w_i$ is not big. 
Applying the Key Lemma to each edge $vw_i$ shows that each $w_i$ must be a
$3^+$-vertex.  Consider the graph $G'$ formed from $G-v$ by adding a path of
length two between each pair of neighbors of $v$. (Since each $w_i$
is not big, we have $\Delta(G')= \Delta(G)$.) Since $G'$ has fewer
$3^+$-vertices, by minimality $G'$ has a good order $\sigma'$, and $\sigma'$ 
also is a good order for $G$ of $V(G)-v$.  Since each neighbor of $v$ is
small, $\card{N^2(v)}\leq 3\cdot9$.  So appending $v$ to $\sigma'$ gives a
good order for $G$, which is a contradiction. 
\end{proof}

\begin{lem}
If a $3$-face $f$ is incident with a $3$-vertex $v$ and at most one big vertex, then the neighbor of $v$ that is not on $f$ must be a big vertex. 
\label{rc-4-3face3s}
\end{lem}
\begin{proof}
Let $v$ be a $3$-vertex on a $3$-face $vw_1w_2$ and let $x$ be the neighbor of
$v$ that is not on $vw_1w_2$.  Suppose to the contrary that both $w_1$ and $x$
are not big.  Applying the Key Lemma to edge $vx$ shows that $x$ is a
$3^+$-vertex.  Consider the graph $G'$ formed from $G-v$ by adding paths
of length two between $x$ and $w_1$ and also between $x$ and $w_2$.  
So, $G'$ has fewer $3^+$-vertices than $G$.  By minimality, $G'$ has a good
order, $\sigma'$, which also is a good order for $G$ of $V(G)-v$.  Since $v$
has at most one big neighbor, $\card{N^2(v)}\leq \Delta(G)+16$.  So appending
$v$ to $\sigma'$ gives a good order for $G$, which is a contradiction. 
\end{proof}

\subsection{Discharging}

We use the initial charges $d(v)-4$ for each vertex $v$ and $\ell(f)-4$ for
each face $f$.  Note that, by Euler's formula, the sum of these initial charges
is $-8$.  
Using the structural lemmas in Section~\ref{sec:reducibles}, we redistribute
this charge so that each vertex and face ends with nonnegative charge.  However,
this gives a contradiction, since a sum of nonnegative numbers cannot equal $-8$.  To
redistribute charge, we use the following six discharging rules, applied in
succession.  (See Figure~\ref{fig-rules} for an illustration of the discharging
rules.)

\begin{enumerate}[(R1)]
\item Each edge takes $1\over 5$ from each incident $5^+$-face and 
$1\over 10$ from each incident big vertex\footnote{A cut-edge takes $2\over 5$
from its incident face.}.

\item If edge $vw$ is incident to a 3-face $f$, then $vw$ gives all its charge
(received by (R1)) to $f$.  Otherwise, $vw$ distributes its charge equally 
among incident vertices $x$ where $d(x)=\min\{d(v),d(w)\}$.

\item Each big vertex gives $1\over 2$ to each neighbor.

\item Each $3$-vertex, $4$-vertex, and small vertex gives $3\over 5$ to each
2-neighbor.
If either $v$ is a $4$-vertex with at least two big neighbors or $v$ is a small
vertex, then $v$ gives $1\over 2$ to each incident 3-face that is incident with
a vertex other than $v$ that is not big.

\item Assume vertices $v$ and $w$ are big and the edge $vw$ lies on a 3-face
$vwx$.  If $x$ is a $4^-$-vertex, then all charge given from $v$ to $w$ (and
vice versa) by (R3) continues on to $x$.  If $x$ is a $5^+$-vertex, then all
charge given from $v$ to $w$ (and vice versa) by (R3) continues on to face
$vwx$.

\item If a $3$-vertex has an incident 3-face $f$ with negative charge, then $v$
gives its excess charge to $f$.  
\end{enumerate}

\begin{figure}[h]
	\begin{center}
  \includegraphics[scale=1]{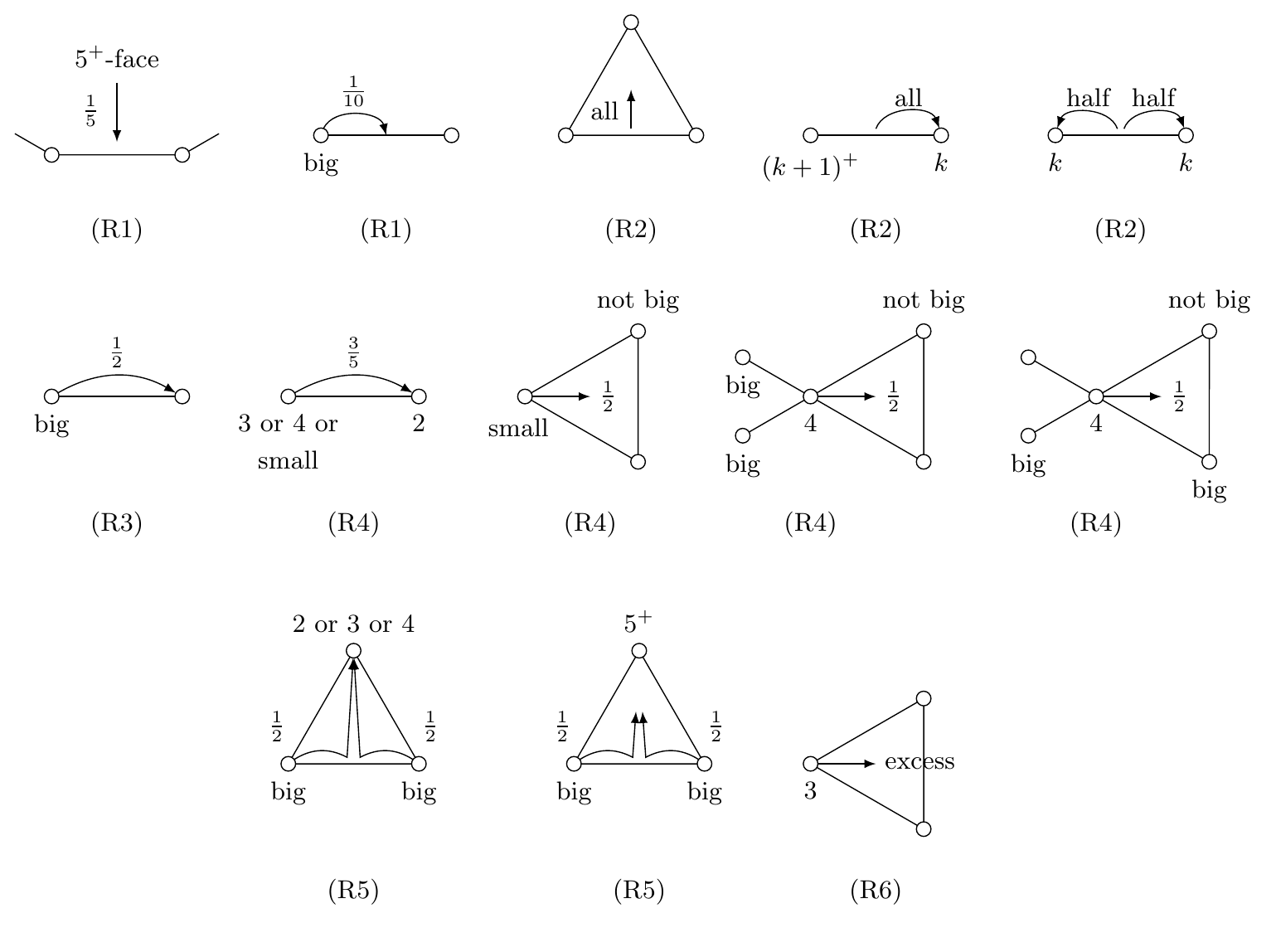}
  \caption{An illustration of the discharging rules.}
  \label{fig-rules}
	\end{center}
\end{figure}

Now we show that each vertex and face ends with nonnegative charge, which yields
the desired contradiction.

Each $5^+$-face $f$ ends with charge
$\ell(f)-4-\frac15\ell(f)=\frac45\ell(f)-4\ge 0$.
Each edge receives charge by (R1) and gives it all away by (R2), so ends with
0.  Consider a big vertex $v$.  For each of its neighbors $w$, the charge that
$v$ gives to $vw$ by (R1) is $1\over 10$ and to $w$ by (R3) is $1\over 2$, for
a total of $3\over 5$. So $v$ ends with at least $d(v)-4-\frac35d(v) =
\frac25d(v)-4$; this is nonnegative, since $d(v)\ge 10$. 

Consider a small vertex $v$.  Let $n_2(v)$\aside{$n_2(v)$, $f_3(v)$} denote the
number of 2-neighbors of $v$ and $f_3(v)$ the number of 3-faces incident with
$v$ that are not incident with two big neighbors of $v$ (that is, 3-faces that
get $1\over 2$ from $v$).  By (R4), $v$ gives away $\frac35n_2(v)+\frac12f_3(v)$.  
Let $w_1,\ldots, w_{d(v)}$ denote the neighbors of $v$.
Suppose that $v$ has a 2-neighbor.  
By the Key Lemma, $v$ has at least two big neighbors, so $n_2(v)\le d(v)-2$.
By Lemma~\ref{rc-3face-2}, since $v$ is small, if $w_i$ is a 2-vertex, then
neither face incident with $vw_i$ receives charge from $v$.
Thus, $n_2(v)+f_3(v)\le d(v)$.  By a more careful
analysis, we will show that $n_2(v)+f_3(v)\le d(v)-2$.  If $f_3(v)=0$, then this
inequality follows from $n_2(v)\le d(v)-2$ above.  So assume $f_3(v)\ge 1$
and that $v$ sends charge to some 3-face $f''$ with boundary that includes
$w_2vw_3$.  Since $G$ has no 4-cycles, neither the face preceding $f''$ around
$v$ nor the face following $f''$ is a 3-face, so neither of them receives
charge from $v$; call these faces $f'$ and $f'''$.  Further, some
$w_j$ other than $w_2$ and $w_3$ is big.  Thus, we can pair the neighbors of $v$
other than $w_2$, $w_3$, and $w_j$ with the faces other than $f'$, $f''$, and
$f'''$ such that each face is paired with an incident vertex and at most one
element in each pair gets charge from $v$.  This proves $n_2(v)+f_3(v)\le
d(v)-2$.
So, $v$ ends with at least
$d(v)-4+2({1\over 2})-\frac35(d(v)-2)=\frac25d(v)-\frac95$; this is positive,
since $d(v)\ge 5$.
Now instead assume that $v$ has no 2-neighbors.  Since $G$ has no 4-cycles,
$f_3(v)\le \left\lfloor {d(v)\over2}\right\rfloor$.  
Thus, $v$ ends with at least $d(v)-4-\frac12\left\lfloor
{d(v)\over2}\right\rfloor$; this is nonnegative since
$d(v)\ge 5$.

So, to complete the proof we only need to consider $3$-faces, $2$-vertices,
3-vertices, and 4-vertices.

\begin{clm}
Every $2$-vertex $v$ that is on a $3$-face $vw_1w_2$ ends with nonnegative 
charge. 
\end{clm}
\begin{clmproof}
By Lemma~\ref{rc-3face-2}, both $w_1$ and $w_2$ must be big.
By (R3), $v$ gets $1\over 2$ from each of $w_1$ and $w_2$.
And by (R5), $v$ gets another $2({1\over 2})$.
So $v$ ends with $2-4+4({1\over 2})=0$. 
\end{clmproof}

\begin{clm}
Every $2$-vertex $v$ that is not on a $3$-face ends with nonnegative charge. 
\end{clm}
\begin{clmproof}
Let $w_1$ and $w_2$ be the neighbors of a $2$-vertex $v$.  It suffices to show
that $v$ gets total charge at least 1 from $w_1$ and $vw_1$, since by symmetry
it also gets at least 1 from $w_2$ and $vw_2$, so $v$ ends with at least
$2-4+2(1)=0$.
Applying the Key Lemma to $vw_1$ shows that $w_1$ either is big or is a
$3^+$-vertex with two
big neighbors.  By (R1), $vw_1$ gets $2\over 5$ from incident faces and by (R2)
$vw_1$ gives all this charge to $v$.  So we only need to show that $v$
gets at least $3\over 5$ from $w_1$.  
If $w_1$ is a $3^+$-vertex that is not
big, then $w_1$ gives $3\over 5$ to $v$ by (R4).
If $w_1$ is big, then it gives $v$ charge $1\over 2$ by (R3), and
gives edge $vw_1$ an extra $1\over 10$ by (R1), and all this charge
 goes to $v$ by (R2).  Thus, $v$ gets $3\over 5$, as desired. 
%
\end{clmproof}

\begin{clm}
Every $3$-vertex $v$ ends with nonnegative charge. 
\end{clm}
\begin{clmproof}
By Lemma~\ref{rc-3-sss}, $v$ has a big neighbor $w$.  

First suppose that $v$ does not have a 2-neighbor.
If $vw$ is not on a 3-face, then by (R3) $w$ gives $v$ charge $1\over 2$, and by
(R2) edge $vw$ gives $v$ charge ${2\over 5}+{1\over 10}$.  
So $v$ ends with at least $3-4+{1\over 2}+{2\over 5}+{1\over
10}=0$.  So assume $v$ is on a 3-face and $vz$ is on a 3-face for every big
neighbor $z$ of $v$.  By Lemma~\ref{rc-4-3face3s}, vertex $v$ has at least
two big neighbors, say $w_1$ and $w_2$.  Since each of $vw_1$ and $vw_2$ must
be on a 3-face, and $v$ has only a single incident 3-face, it must be
$vw_1w_2$.  Now, by (R3) and (R5), $v$ gets at least $4({1\over2})$.  So $v$
ends (R5) with at least $3-4+4({1\over 2})>0$.

Now assume that $v$ has a 2-neighbor $x$, which gets $3\over 5$ from $v$ by
(R4).  Since $d(v)=3$ and $d(x)=2$, Lemma~\ref{rc-3face-2} implies that $vx$ is
not incident to any 3-face.
Applying the Key Lemma to $vx$ shows that $v$ has two big neighbors,
$w_1$ and $w_2$.  If $vw_1w_2$ is a 3-face, then each of $w_1$ and $w_2$ gives
${1\over 2}+{1\over 2}$ to $v$, by (R3) and (R5).  So $v$ ends with at least
$3-4-{3\over 5}+4({1\over 2})>0$.  If $vw_1w_2$ is not a 3-face, then each of
$vw_1$ and $vw_2$ gives $1\over 2$ to $v$, by (R1) and (R2).  So $v$ ends with
at least $3-4-{3\over 5}+2({1\over 2})+2({1\over 2})>0$.
\end{clmproof}

\begin{clm}
Every $4$-vertex $v$ ends with nonnegative charge. 
\end{clm}
\begin{clmproof}
Let $n_2(v)$ and $f_3(v)$ denote, respectively, the numbers of 2-neighbors and
incident 3-faces that get charge from $v$ by (R4).  

Suppose $v$ has no $2$-neighbor.  If $v$ gives no charge to incident $3$-faces
by (R4), then $v$ gives no charge at all, so $v$ ends with at least $4-4=0$.
If $v$ does give charge to an incident $3$-face by (R4), then (R4) implies that
$v$ has two big neighbors; by (R3), each big neighbor gives $v$ charge $1\over
2$.  Since $G$ has no 4-cycles, $v$ gives charge to at most two $3$-faces. 
So $v$ ends with at least $4-4+2({1\over 2})-2({1\over 2})=0$.

So assume $v$ has a $2$-neighbor, $u$.  Applying the Key Lemma to $uv$ shows
that $v$ has two big neighbors, $w_1$ and $w_2$; by (R3) each $w_i$ gives $v$
charge $1\over 2$.  If $vw_i$ is not on a $3$-face, for some $w_i$, then by
(R2) $vw_i$ gives $v$ charge ${2\over 5}+{1\over 10}$.  Thus, $v$ ends with at
least $4-4+2({1\over 2})+({2\over 5}+{1\over 10})-2\cdot{3\over5}>0$.  
So we assume that each $vw_i$ is on a 3-face.  Since $v$ has a 2-neighbor (which
is not on a 3-face with $v$, by Lemma~\ref{rc-3face-2}), and $G$ has no
4-cycles, $v$ has at most one incident 3-face.  Since $vw_1$ and $vw_2$ are both
on 3-faces, the 3-face must be $vw_1w_2$.  Because $w_1$ and $w_2$ are both big,
$v$ gives no charge to $vw_1w_2$.  So $v$ ends with at least $4-4+4({1\over
2})-2({3\over 5})>0$.
\end{clmproof}

\begin{clm}
Every $3$-face ends with nonnegative charge. 
\end{clm}
\begin{clmproof}
Let $f=v_1v_2v_3$ be a $3$-face, where $d(v_1)\leq d(v_2)\leq d(v_3)$. 
By (R1) each of $v_1v_2, v_2v_3, v_3v_1$ gets $1\over 5$ from its incident
$5^+$-face, and by (R2) all of this charge goes to $f$.
If $f$ has two incident big vertices, then by (R1) edges $v_1v_2, v_2v_3,
v_3v_1$ get in total an additional ${4\over 10}$.  So $f$ ends with at least
$3-4+{3\over 5}+{4\over10}=0$.  If $v_1$ is a $2$-vertex, then $v_2$ and
$v_3$ are both big, by Lemma~\ref{rc-3face-2}, and we are done, as above.
So assume that $v_1$ is a $3^+$-vertex, and $v_2$ is not big. 
If some $v_i$ is a small vertex or a 4-vertex with two big neighbors (which, by
assumption, are not both incident to $f$), then $v_i$
gives $1\over 2$ to $f$ by (R4), so $f$ ends with at least $3-4+3({1\over
5})+{1\over 2}>0$.  So we assume that $f$ has at most one incident big vertex,
and has no incident small vertex, and no incident 4-vertex with two big
neighbors.  Applying
the Key Lemma to $v_1v_2$ shows that $f$ must have an incident big vertex. 
Otherwise $v_1$ and $v_2$ are each $4^-$-vertices with at most one big
neighbor, a contradiction.
Thus, we can assume that $f$ has exactly one incident big vertex, and has no
incident 2-vertex, small vertex, or 4-vertex with two big neighbors.

So assume that $v_3$ is big and that $v_1$ and $v_2$ are each either a 3-vertex
or else a 4-vertex with no big neighbor other than $v_3$.  Applying the Key
Lemma to $v_1v_2$ shows that $v_1$ must be a 3-vertex.  Furthermore, at least
one of $v_1$ and $v_2$ is a 3-vertex with a big neighbor $w$ not on $f$; by
symmetry, assume this is $v_1$.  By (R3), $w$ and $v_3$ each give $v_1$ charge
$1\over 2$.  Since edge $wv_1$ is not on a $3$-face, by (R2) it gives $v_1$
charge ${2\over 5}+{1\over 10}$.  So $v_1$ finishes (R5) with at least
$3-4+2({1\over 2})+{2\over 5}+{1\over 10}={1\over 2}$; by (R6) all of this
charge continues on to $f$.  So $f$ ends with at least $3-4+3({1\over
5})+{1\over 2}>0$.
\end{clmproof}

This completes the proof of Theorem~\ref{4-cycle-thm}.  \hfill $\square$\\

For completeness, we conclude this section with the definitions of
Alon--Tarsi number, paint number and correspondence chromatic number,
and the corollary that bounds these parameters for planar graphs with
no 4-cycles.  To denote the list-chromatic number of a graph $G$, we write
$\chil(G)$.

An \Emph{eulerian digraph} is one in which each vertex has indegree equal to
outdegree.  For a digraph $D$, let $EE(D)$ and $EO(D)$ denote the numbers of
eulerian subgraphs of $D$ in which the number of edges is even and odd,
respectively.  A digraph $D$ is \EmphE{Alon--Tarsi}{-.5cm} if $EE(D)\ne EO(D)$,
and it
is \Emph{$k$-Alon--Tarsi} if also each vertex has outdegree less than $k$.
An \Emph{orientation} of a graph $G$ is formed from $G$ by directing
each edge toward one of its endpoints.  The \Emph{Alon--Tarsi number} of
$G$, denoted
$\AT(G)$, is the smallest $k$ such that some orientation of $G$ is
$k$-Alon--Tarsi.  Note that every acyclic orientation $D$ is Alon--Tarsi, since
$EE(D)=1\ne 0=EO(D)$; the only eulerian subgraph of $D$ is the spanning
edgeless graph.  Suppose that $G$ has degeneracy $k$, and $\sigma$ is a vertex
ordering witnessing this.  By orienting each edge toward its endpoint that
appears earlier in $\sigma$, we conclude that $\AT(G)\le k+1$.

The paint number is defined using a two-player game. At round $i$, one
player (Lister) chooses a set $S_i$ of vertices and the other one
(Painter) answers by coloring an independent subset of $S_i$ with
color $i$. The winning conditions depend on a fixed integer $k$:
Lister wins if he presents a vertex on $k$ rounds but Painter never
colors it. Otherwise, Painter wins. The \Emph{paint number}
$\chi_p(G)$ is the smallest integer $k$ such that Painter has a
winning strategy with parameter $k$. This problem can be seen as a
generalization of list coloring, where the lists are not all known at
the beginning of the coloring process (take $S_i$ as the set of
vertices whose lists contain color $i$). As shown by
Schauz~\cite{schauz}, each $k$-Alon--Tarsi graph is
$k$-paintable. Thus, every $k$-degenerate graph $G$ satisfies
$\chi_p(G)\le \AT(G)\le k+1$.

Given a graph $G$ and a function $f:V(G)\to\mathbb{N}$, an
\emph{$f$-correspondence assignment} $C$ is given by a matching $C_{vw}$,
 for each $vw\in E(G)$, between $\{v\}\times\{1,\ldots,f(v)\}$ and
$\{w\}\times\{1,\ldots,f(w)\}$. We say that each vertex $x$ has \emph{$f(x)$
available colors}. A \emph{$k$-correspondence assignment} is an
$f$-correspondence assignment where $f(v)=k$ for all $v\in V(G)$.
Given an $f$-correspondence assignment $C$, a \Emph{$C$-coloring} is a
function $\varphi:V(G)\to\mathbb{N}$ such that $\varphi(v)\leqslant f(v)$ for
each $v\in V(G)$, and, for each edge $vw\in E(G)$, the pairs $(v,\varphi(v))$
and $(w,\varphi(w))$ are nonadjacent in $C_{vw}$.
The \Emph{correspondence chromatic number} of $G$, denoted $\chi_{corr}(G)$,
is the least integer $k$ such that, for every $k$-correspondence assignment $C$
of $G$, graph $G$ admits a $C$-coloring. Note that if $G$ is $k$-degenerate,
then coloring greedily in some order witnessing this shows that
$\chi_{corr}(G)\le k+1$.  Thus, we have the following corollary of
Theorem~\ref{4-cycle-thm}.

\begin{cor}
  If $G$ is a planar graph with no 4-cycles, then
  $\chi_{corr}(G^2)\le \Delta(G)+73$, $\chi_p(G^2)\leqslant \Delta(G)+73$,
  and $\AT(G^2)\le \Delta(G)+73$.
\label{maincor}
\end{cor}

\section{Graphs with no 4-cycles and $\Delta$ large}
\label{sec3}

In this section we show that the upper bounds in Corollary~\ref{maincor} can be
strengthened to $\Delta(G)+2$ when $\Delta(G)$ is sufficiently large.
Initially, we just prove this upper bound for $\AT(G^2)$, which also implies it
for $\chi_p(G^2)$ and $\chil(G^2)$.  In Section~\ref{sec5} we extend this result
to $\chi_{corr}(G^2)$.

\begin{thm}
  \label{thm:main}
  There exists $\Delta_0$ such that if $G$ is a plane graph with no
  4-cycles and with $\Delta(G)\geq \Delta_0$, then $G^2$ is
  $(\Delta(G)+2)$-choosable. In fact, $\chi_p(G^2)\le\AT(G^2)\le \Delta(G)+2$.
\end{thm}

Let $\Delta_0=23769500^2=564989130250000$ and fix $k\ge \Delta_0$.  We
prove by contradiction that if $G$ is a plane graph with no 4-cycles
and with $\Delta(G)\le k$, then $G^2$ is $(k+2)$-choosable.  (By
\Emph{plane graph}, we mean a planar graph with a fixed embedding in
the plane. In particular, the neighborhood of each vertex is naturally
endowed with a cyclic ordering.) For ease of exposition, we present
the proof only for choosability, although it also works for
paintability and Alon--Tarsi orientations.  Most of the reducible
configurations rely only on degeneracy, though at one point we use the
kernel lemma.

Assume the theorem is false and let $G$ be a counterexample that minimizes
$|E(G)|+|V(G)|$.  Let $L$ be an assignment of lists of size $k+2$ to
the vertices of $G$ such that $G^2$ has no $L$-coloring.
Throughout Section~\ref{sec3} we prove several structural
lemmas, which ultimately lead to a contradiction. 
We follow the same general approach as in~\cite{BCP}, which considered
planar graphs with girth at least 5; however, we need new ideas to
handle the presence of triangles.

\subsection{First Reducible Configurations} 

\begin{lem}
\label{sec3:lem1}
The graph $G$ is connected and has minimum degree at least $2$.
\end{lem}

\begin{proof}
  Note that $G$ is connected, since otherwise one of its components is
  a smaller counterexample.  Now assume there exists a 1-vertex
  $v\in V(G)$. By the minimality of $G$, we can $L$-color
  $(G\setminus \{v\})^2$.  Since $|L(v)|=k+2$, and $v$ has at most
  $1+(k-1)$ neighbors in $G^2$, we can color $v$ with a color not used
  on its neighbors in $G^2$, which is a contradiction.
\end{proof}

The next two lemmas essentially show that every vertex of $G$ must
be near a vertex of high degree. To formalize this, we use the
following terminology: a vertex $v\in V(G)$ is \emph{big}\aside{big, small} if
$\deg(v)\ge \sqrt{k}$ and \emph{small} otherwise. Denote by $B$ and $S$
the sets of big and small vertices. To refine the set $S$, we write 
$S_i$ for the set of small vertices with exactly $i$ big neighbors. 

\begin{rem}
In our figures in the rest of the paper, we draw small vertices as circles, and
big vertices as
squares.  Further, we use black circles for vertices with all neighbors shown.
So a white vertex could have more neighbors than those shown; in fact, it could
also have edges (that are not drawn) to other vertices that are shown.  
For example, Figure~\ref{fig:first_red} shows the configurations
forbidden by Lemma~\ref{lem:first_red}. 
\end{rem}

\begin{figure}[!ht]
  \centering
  \begin{tikzpicture}[
      b/.style={fill=black,minimum size = 5pt,rectangle,inner sep=1pt},
      s/.style={fill=black,minimum size = 5pt,ellipse,inner sep=1pt},
      s2/.style={draw=black,minimum size = 5pt,ellipse,inner sep=1pt},
      s1/.style={draw=black,minimum size = 5pt,ellipse,inner sep=1pt,fill=black!20},
      e/.style={fill=white,minimum size = 0pt,ellipse,inner sep=0pt},
    ]
    \node[s2] (u1) at (0,0.5)  {};
    \node[s2] (u2) at (0,-0.5)  {};
    \node[s,label=above:\footnotesize{$v$}] (u) at (1,0)  {};
    \node[s,label=above:\footnotesize{$w$}] (v) at (2,0) {};
    \node[s2] (v1) at (3,0.5)  {};
    \node[s2] (v2) at (3,-0.5)  {};
    \draw (u1) -- (u) node[midway] (uu1){};
    \draw (u) -- (v);
    \draw (v) -- (v1) node[midway] (vv1){};
    \draw (u2) -- (u)node[midway] (uu2){};
    \draw (v2) -- (v)node[midway] (vv2){};
    \draw [dotted] (uu1) -- (uu2);
    \draw [dotted] (vv1) -- (vv2);
    \tikzset{yshift=-1cm}
    \node[e,label=above:{}] (x) at (0.5,-.250)  {};
    \node[s,label=above:\footnotesize{$v$}] (u) at (1,-.250)  {};
    \node[s,label=above:\footnotesize{$w$}] (v) at (2,-.25) {};
    \node[s2] (v1) at (3,-.250)  {};
    \draw (x)-- (u) -- (v);
    \draw (v) -- (v1);
  \end{tikzpicture}
  \caption{Forbidden configurations of Lemma~\ref{lem:first_red}.}
  \label{fig:first_red}
\end{figure}
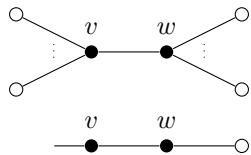

\begin{lem}
  \label{lem:first_red}
  For each edge $vw\in E(G)$, either $v\in N[B]$ or $w\in
  N[B]$. Further, if $\deg(v)=\deg(w)=2$, then $v,w\in N[B]$.
\end{lem}

\begin{proof}
  Assume to the contrary that some edge $vw$ has
  $v,w\notin N[B]$. By minimality, we can $L$-color 
  $(G-vw)^2$. We uncolor $v$ and $w$.
  Since $v,w\notin N[B]$, both $v$ and $w$ have less than
  $\sqrt{k}\times\sqrt{k}$ colored neighbors in $G^2$.  Since
  $|L(v)|=|L(w)|=k+2$, we can find distinct available colors for $v$
  and $w$.

Suppose instead that $d(v)=d(w)=2$ and $v\in N[B]$ and $w\notin N[B]$.
Again, by minimality we $L$-color $(G-vw)^2$, then uncolor $v$ and $w$.
  Now $v$ has at most $k+1$ colored neighbors
  in $G^2$, so $v$ has an available color. 
As before, we can color $w$.
This gives an $L$-coloring for $G^2$, a contradiction.
\end{proof}

\begin{lem}
  \label{lem:notriangle}
  If $vw$ is an edge with $\deg(v)=\deg(w)=2$, then $v$ and $w$
  have no common neighbor.
\end{lem}

\begin{proof}
  Assume there exists a triangle $vwx$ with $\deg(v)=\deg(w)=2$. 
 By minimality, we can $L$-color $(G\setminus\{v,w\})^2$. Both $v$ and $w$
  have $\deg(x)-1\leq k-1$ colored neighbors in $G^2$. So
  $v$ and $w$ each have at least $3$ available colors, and thus we
  can color them both.
\end{proof}

\begin{lem}
  \label{lem:config++}
  Let $vx_1x_2$ be a triangle of $G$ such that some vertex
  $w\in S\setminus\{v,x_1,x_2\}$ has a common 2-neighbor with $x_1$.
  If either (a) $d(x_2)=2$ or (b) $d(x_2)=3$ and $w$ and $x_2$ have a
  common 2-neighbor, then $d(x_1)\ge 4$.
\end{lem}

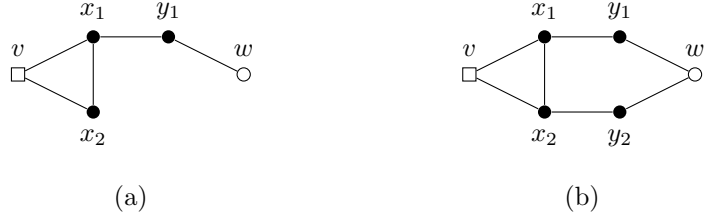
\begin{figure}[!ht]
  \centering
  \begin{tikzpicture}[xscale=-1,
      b/.style={fill=black,minimum size = 5pt,rectangle,inner sep=1pt},
      bw/.style={draw=black,minimum size = 5pt,rectangle,inner sep=1pt},
      s/.style={fill=black,minimum size = 5pt,ellipse,inner sep=1pt},
      s2/.style={draw=black,minimum size = 5pt,ellipse,inner sep=1pt},
      s1/.style={draw=black,minimum size = 5pt,ellipse,inner sep=1pt,fill=black!20},
      e1/.style={draw=white,minimum size = 5pt,ellipse,inner sep=1pt},
    ]
    \node[s2,label=above:\footnotesize{$w$}] (u) at (0,-0.5)  {};
    \node[s,label=above:\footnotesize{$x_1$}] (x3) at (2,0) {};
    \node[s,label=below:\footnotesize{$x_2$},below of=x3] (x2) {};
    \node[s,label=above:\footnotesize{$y_1$}] (y3) at (1,0) {};
    \node[s,label=below:\footnotesize{$y_2$},below of=y3] (y2) {};
    \node[bw,label=above:\footnotesize{$v$}] (v) at (3,-0.5) {};
    \draw (v) -- (x2) -- (y2) -- (u);
    \draw (v) -- (x3) -- (y3) -- (u);
    \draw (x3) -- (x2);
    \node[e1] at (1.5,-2.15) {\footnotesize(b)};

    \tikzset{xshift=6cm}
    \node[s2,label=above:\footnotesize{$w$}] (u) at (0,-0.5)  {};
    \node[s,label=above:\footnotesize{$x_1$}] (x3) at (2,0) {};
    \node[s,label=below:\footnotesize{$x_2$},below of=x3] (x2) {};
    \node[s,label=above:\footnotesize{$y_1$}] (y3) at (1,0) {};
    \node[bw,label=above:\footnotesize{$v$}] (v) at (3,-0.5) {};
    \draw (v) -- (x2) -- (x3);
    \draw (v) -- (x3) -- (y3) -- (u);
    \node[e1] at (1.5,-2.15) {\footnotesize(a)};

  \end{tikzpicture}
  \caption{Forbidden configurations of Lemma~\ref{lem:config++}.}
  \label{fig:config++}
\end{figure}

\begin{proof}
  Let $y_1$ and $y_2$ denote the 2-neighbors of $w$ common with $x_1$
  and $x_2$ if they exist (in Case (a), only $y_1$ is defined).
  Assume that $\deg(x_1)=3$. 
  If $vw\in E(G)$, then $wvx_1y_1$ is a 4-cycle in $G$, a contradiction.  So
$vw\notin E(G)$.  By assumption, $w\notin\{v,x_1,x_2,y_1\}$.  So if $wx_2\in
E(G)$, then $wx_2x_1y_1$ is a 4-cycle in $G$, again a contradiction.  Thus,
$wx_2\notin E(G)$.  Since $d(x_1)=3$ and $v,x_2,y_1\in N(x_1)$, we must have
$w\notin N(x_1)$.  Since $N(y_1)=\{x_1,w\}$, also $vy_1\notin E(G)$.
So in both cases $wx_1,wx_2,wv,vy_1\notin E(G)$.  And in (b)
also $vy_2\notin E(G)$.

  Let $S=\{x_1,x_2,y_1\}$ in Case (a), and $S=\{x_1,x_2,y_1,y_2\}$ in
  Case (b). By minimality, we $L$-color $(G\setminus S)^2$.
  For each $i\in \{1,2\}$, the number of colored neighbors in $G^2$ of
  $x_i$ is at most:
\[|\{v,w\}|+ |N(v)\setminus\{x_1,x_2\}|\leq 2+(k-2)=k.\]
Thus, $x_1$ and $x_2$ both have at least 2 available colors, so we can
color them.
  Further, for each $i\in\{1,2\}$, the number of colored neighbors of $y_i$ 
is at most \[|\{v,w,x_1,x_2\}|+|N(w)\setminus\{y_i\}|\leq
    4+\sqrt{k}-1=\sqrt{k}+3.\]
  Therefore, $y_1$ and $y_2$ (if defined) both have $k-\sqrt{k}-1$
  available colors. Since $k$ is large enough, we can color them to
  get an $L$-coloring for $G^2$, a contradiction.
\end{proof}

We combine Lemmas~\ref{lem:first_red} and \ref{lem:config++} to prove
the reducibility of the bigger configuration shown in
Figure~\ref{fig:smallfacesaux}.

\begin{figure}[!ht]
  \centering
  \begin{tikzpicture}[
      b/.style={draw=black,minimum size = 5pt,rectangle,inner sep=1pt},
      s/.style={fill=black,minimum size = 5pt,ellipse,inner sep=1pt},
      s2/.style={draw=black,minimum size = 5pt,ellipse,inner sep=1pt},
      s1/.style={draw=black,minimum size = 5pt,ellipse,inner sep=1pt,fill=black!20},
    ]
    \node[s2,label=below:\footnotesize{$w$}] (u) at (0,0)  {};
    \node[s2,label=left:\footnotesize{$x_3$}] (x3) at (0,2) {};
    \node[s2,label=right:\footnotesize{$\!\!x_4$}] (x4) at (1,2) {};
    \node[s2,label=right:\footnotesize{$x_5$}] (x5) at (2,2) {};
    \node[s2,label=right:\footnotesize{$\hspace{-.8cm}x_2$}] (x2) at (-1,2) {};
    \node[s2,label=left:\footnotesize{$\,x_1$}] (x1) at (-2,2) {};
    \node[s2] (x6) at (0.33,2) {};
    \node[s2] (x7) at (0.66,2) {};
    \node[s2] (x8) at (-1.66,2) {};
    \node[s2] (x9) at (1.66,2) {};
    \node[s,label=left:\footnotesize{$y_3$}] (y3) at (0,1) {};
    \node[s,label=right:\footnotesize{$y_4$},right of=y3] (y4) {};
    \node[s,label=right:\footnotesize{$y_5$},right of=y4] (y5) {};
    \node[s,label=left:\footnotesize{$y_2$}] at (-1,1) (y2) {};
    \node[s,label=left:\footnotesize{$y_1$},left of=y2] (y1) {};
    \node[b,label=above:\footnotesize{$v$}] (v) at (0,3) {};
    \draw (v) -- (x1) -- (y1) -- (u);
    \draw (v) -- (x2) -- (y2) -- (u);
    \draw (v) -- (x3) -- (y3) -- (u);
    \draw (v) -- (x4) -- (y4) -- (u);
    \draw (v) -- (x5) -- (y5) -- (u);
    \draw (x6) -- (x7) -- (v) -- (x6);
    \draw (x1) -- (x8) -- (v) -- (x9);
  \end{tikzpicture}
  \caption{A possible configuration of Lemma~\ref{lem:smallfacesaux}.}
  \label{fig:smallfacesaux}
\end{figure}
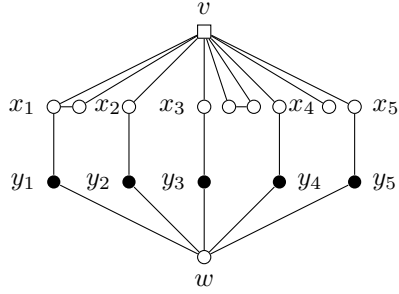

\begin{lem}
  \label{lem:smallfacesaux}
  Fix $v, w\in V(G)$ such that $w\in S$. Then the graph $G$ cannot
  contain distinct vertices $y_1,\ldots,y_5$ that are consecutive
  neighbors of $w$ and that satisfy both conditions below; see
  Figure~\ref{fig:smallfacesaux}.
  \begin{enumerate}
  \item Each $y_i$ has degree two and has a common neighbor $x_i$ with $v$.
  \item For each $i\in\{1,\ldots,4\}$, each vertex inside cycle
$vx_iy_iwy_{i+1}x_{i+1}$ is adjacent to $v$.
  
  \end{enumerate}
\end{lem}

\begin{proof}
We assume that $G$ contains such a configuration and reach a contradiction, by
showing that $G$ contains a configuration forbidden by Lemma~\ref{lem:config++}.
Note that all $x_i$'s are distinct, since $G$ contains no 4-cycle.

Below when we write a statement about $x_i$, we mean that it is true for each
$i\in\{2,3,4\}$.  Since $w\in S$, Lemma~\ref{lem:first_red} implies that
$d(x_i)\ge 3$.  Because $y_1,\ldots,y_5$ are consecutive
neighbors of $w$, vertex $x_i$ is not adjacent to $w$.  Since $G$ has no
4-cycle, $x_i$ has at most one common neighbor with $v$.  Thus $d(x_i)=3$.
Define $z$ so that $N(x_3)=\{v,y_3,z\}$.  If $z\in \{x_2,x_4\}$, then $G$
contains the second configuration in Lemma~\ref{lem:config++}, a contradiction.
If $z$ has a neighbor other than $x_3$ and $v$, then call it $z'$; now $z'$ is
adjacent to $v$ (by hypothesis 2), so $vx_3zz'$ is a 4-cycle, a contradiction.
Thus, $z$ is a 2-vertex with $N(z)=\{x_3,v\}$. 
Now $G$ contains the first configuration in Lemma~\ref{lem:config++}, again a
contradiction.
\end{proof}
  
\subsection{Outline of the proof}
\label{sec4}

Recall that $S$\aside{$S$, $S_i$} is the set of small vertices, and
$S_i$ is the set of small vertices with exactly $i$ big neighbors.
Let \Emph{$G'$} denote the multigraph formed from $G$ by suppressing
every vertex of degree $2$ in $S\setminus N[B]$, and then contracting
every edge between $S_1$ and $B$. (Suppressing a $2$-vertex $v$ means
deleting $v$ and adding an edge between its two neighbors.) Note that
$G'$ may contain loops. For example, there is a loop in $G'$ around a
vertex $u$ if $u$ is a big vertex in $G$ and there is a triangle $uvw$
with $v,w\in S_1$. We say that a vertex of $G$ \emph{disappears} when
constructing $G'$ if it is either a suppressed vertex, or a vertex in
$S_1$.

Let \Emph{$G''$} denote the multigraph formed from $G'$ by removing
every loop, and let \Emph{$G'''$} denote the underlying multigraph of
$G''$, i.e.,~the multigraph formed from $G''$ by deleting the minimal
number of edges to remove all faces of length $2$.  Note that $G'''$
can have parallel edges.  For example, suppose $v$ and $w$ have
parallel edges, say $e_1$ and $e_2$, in $G'$.  If some vertices are
embedded inside and outside of the cycle $e_1e_2$, then in $G'''$
vertices $v$ and $w$ still have parallel edges, with those same
vertices embedded inside and outside of the cycle $e_1e_2$.  However,
$G'''$ cannot have faces of length 2.

An \Emph{$r$-region of $G''$} is a set $\{f_1,\ldots,f_r\}$ of $r$
pairwise distinct faces of length $2$ such that:
\begin{itemize}
\item For $1\leqslant i<r$, $f_i$ shares one edge with $f_{i+1}$. (We
  say that the $f_i$'s are \Emph{consecutive}.)
\item All the $f_i$'s have the same vertices $b_1,b_2$ on their
  boundary, where $b_1$ and $b_2$ are distinct vertices of $B$.
\end{itemize}
Note that each of the faces in an $r$-region is constructed from some
cycle of $G$ when we apply the construction rules above.  By
extension, an \Emph{$r$-region of $G$} is the subgraph of $G$ induced
by the vertices of these cycles, together with those lying on the
inside of those cycles.  (We often simply write \Emph{region}, when
the specific value of $r$ is less important.)  When $R$ is an
$r$-region of $G$, we say that $r$ is the size of $R$, and we denote
by \Emph{$V(R)$} the set of vertices appearing on all faces of $R$,
excluding $b_1$ and $b_2$.

To reach a contradiction, we prove the following two propositions.

\begin{prop}
  \label{prop:bigreg}
  $G$ contains an $r$-region with $r\geqslant\frac{\sqrt{k}}{50}-37$.
\end{prop}

\begin{prop}
  \label{prop:smallreg}
  $G$ does not contain any $r$-region for $r\geq 475353$.
\end{prop}

Our contradiction now comes quickly.  These propositions give that
$\frac{\sqrt{k}}{50}-37 < 475353$.  This inequality implies
$k<23769500^2$, contradicting the hypothesis $k\geq
\Delta_0=23769500^2$.

We will devote a subsection to the proof of each
proposition: Subsection~\ref{subsec:bigreg} for
Proposition~\ref{prop:bigreg} and Subsection~\ref{subsec:smallreg} for
Proposition~\ref{prop:smallreg}. In Subsection~\ref{subsec:struct}, we
prove structural lemmas about the regions in $G$.

\subsection{Structure of Regions}
\label{subsec:struct}
We now classify each edge of $G'$ based on its corresponding path (or
cycle) in $G$.  An edge $e$ in $G'$ \EmphE{corresponds to a path or
  cycle $x_1\cdots x_n$ in $G$}{-.4cm} if $e=x_1x_n$ and for each
$i\in\{2,\ldots, n-1\}$, one of the following holds:
\begin{itemize}
\item $x_i$ is a 2-vertex in $G$ and $x_{i-1},x_{i+1}\in S$, or
\item $x_i\in S_1$ and either $x_{i-1}$ or $x_{i+1}$ lies in
  $B$. 
\end{itemize}

Due to the construction of $G'$, for every loop (resp.~non-loop edge)
$e$ of $G'$, there is a unique cycle (resp.~path) $x_1\cdots x_n$ in
$G$ corresponding to $e$ (with possibly $n=2$). Note that we used here
that the suppressed 2-vertices are not in $N[B]$, hence every
contracted edge (between $S_1$ and $B$) is between two adjacent
vertices in $G$.

The following lemma ensures that every edge (resp.~loop) of $G'$
corresponds to a short path (resp.~cycle) of $G$. It also gives a
classification of all the possible such paths (resp.~cycles), depicted
in Figure~\ref{fig:types}.

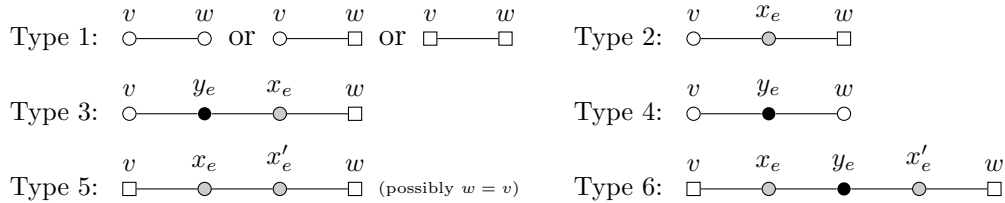
\begin{figure}[!ht]
  \centering
  \begin{tikzpicture}[
      b/.style={draw=black,minimum size = 5pt,rectangle,inner sep=1pt},
      s/.style={fill=black,minimum size = 5pt,ellipse,inner sep=1pt},
      s2/.style={draw=black,minimum size = 5pt,ellipse,inner sep=1pt},
      s1/.style={draw=black,minimum size = 5pt,ellipse,inner sep=1pt,fill=black!20},
    ]
    \node at (-1,0) {\footnotesize Type 1:};
    \node[s2,label=above:\footnotesize{$v$}] (u) at (0,0)  {};
    \node[s2,label=above:\footnotesize{$w$}] (v) at (1,0) {};
    \draw (u) -- (v);
    \node at (1.5,0) {or};
    \node[s2,label=above:\footnotesize{$v$}] (u) at (2,0)  {};
    \node[b,label=above:\footnotesize{$w$}] (v) at (3,0) {};
    \draw (u) -- (v);
    \node at (3.5,0) {or};
    \node[b,label=above:\footnotesize{$v$}] (u) at (4,0)  {};
    \node[b,label=above:\footnotesize{$w$}] (v) at (5,0) {};
    \draw (u) -- (v);
    \tikzset{yshift=-1cm}
    \node at (-1,0) {\footnotesize Type 3:};
    \node[s2,label=above:\footnotesize{$v$}] (u) at (0,0)  {};
    \node[s,label=above:\footnotesize{$y_e$}] (ye) at (1,0) {};
    \node[s1,label=above:\footnotesize{$x_e$}] (xe) at (2,0) {};
    \node[b,label=above:\footnotesize{$w$}] (v) at (3,0) {};
    \draw (u) -- (ye) -- (xe) -- (v);
    \tikzset{yshift=-1cm}
    \node at (-1,0) {\footnotesize Type 5:};
    \node[b,label=above:\footnotesize{$v$}] (u) at (0,0)  {};
    \node[s1,label=above:\footnotesize{$x_e$}] (xe) at (1,0) {};
    \node[s1,label=above:\footnotesize{$x'_e$}] (x'e) at (2,0) {};
    \node[b,label=above:\footnotesize{$w$}] (v) at (3,0) {};
    \node at (4.25,0) {\tiny (possibly $w=v$)};
    \draw (u) -- (xe) -- (x'e) -- (v);
    \tikzset{xshift=7.5cm,yshift=2cm}
    \node at (-1,0) {\footnotesize Type 2:};
    \node[s2,label=above:\footnotesize{$v$}] (u) at (0,0)  {};
    \node[s1,label=above:\footnotesize{$x_e$}] (xe) at (1,0) {};
    \node[b,label=above:\footnotesize{$w$}] (v) at (2,0) {};
    \draw (u) -- (xe) -- (v);
    \tikzset{yshift=-1cm}
    \node at (-1,0) {\footnotesize Type 4:};
    \node[s2,label=above:\footnotesize{$v$}] (u) at (0,0)  {};
    \node[s,label=above:\footnotesize{$y_e$}] (ye) at (1,0) {};
    \node[s2,label=above:\footnotesize{$w$}] (v) at (2,0) {};
    \draw (u) -- (ye) -- (v);
    \tikzset{yshift=-1cm}
    \node at (-1,0) {\footnotesize Type 6:};
    \node[b,label=above:\footnotesize{$v$}] (u) at (0,0)  {};
    \node[s1,label=above:\footnotesize{$x_e$}] (xe) at (1,0) {};
    \node[s,label=above:\footnotesize{$y_e$}] (ye) at (2,0) {};
    \node[s1,label=above:\footnotesize{$x'_e$}] (x'e) at (3,0) {};
    \node[b,label=above:\footnotesize{$w$}] (v) at (4,0) {};
    \draw (u) -- (xe) -- (ye) -- (x'e) -- (v);
  \end{tikzpicture}
  \caption{The six types of paths/cycles in $G$ that create edges in
    $G'$ (gray vertices lie in $S_1$).}
  \label{fig:types}
\end{figure}

\begin{lem}
  \label{lem:struct_edge}
  Each edge $e=vw$ of $G'$ corresponds to a path or a cycle in $G$ for
  which exactly one of the following six conditions holds (up to
  exchanging $v$ with $w$).  If $e$ satisfies condition $i$ below (for
  some $i\in [6]$), then we say that $e$ has \emph{type $i$}\aside{type $i$}.  
  If $v\in S$, then $e$ has one of types 1--4.  If $e$ is a loop of $G'$,
  then $e$ has type $5$.  Finally if $v,w\in B$, then $e$ has type 1,
  5, or 6.
  \begin{enumerate}
  \item $e\in E(G)$.
  \item $w\in B$ and $e$ corresponds to a path $vx_ew$ in $G$ with
    $x_e\in S_1$.
  \item $w\in B$ and $e$ corresponds to a path $vy_ex_ew$ in $G$ with
    $x_e\in S_1$ and $\deg(y_e)=2$.
  \item $w\in S$ and $e$ corresponds to a path $vy_ew$ in $G$ with
    $\deg(y_e)=2$.
  \item $e$ corresponds to a path or cycle $vx_ex'_ew$ in $G$ with
    $x_e,x'_e\in S_1$.
  \item $e$ corresponds to a path $vx_ey_ex'_ew$ in $G$ with
    $x_e,x'_e\in S_1$ and $\deg(y_e)=2$.
  \end{enumerate}
\end{lem}

\begin{proof}
  Due to the construction of $G'$, each edge $e$ in $G'$ between $v$
  and $w$ comes from a path (or cycle) $P_e$ in $G$ between $v$ and
  $w$. In particular, every internal vertex of $P_e$ is either a
  2-vertex in $S\setminus N[B]$ or a vertex of $S_1$ which is preceded
  or followed in $P_e$ by a big vertex. This implies that each
  internal vertex of $P_e$ is small, and that the only vertices of
  $P_e$ that can be big are $v$ and $w$.

  By Lemma~\ref{lem:first_red}, no two consecutive vertices of $P_e$
  are suppressed. This implies that $P_e$ has length at most four.
  \begin{itemize}
  \item If $P_e$ has length one, then $e$ has type 1.
  \item If $P_e$ has length two, then we have $v\neq w$ since $G$ is
    simple. Denote by $x$ the middle vertex of $P_e$. We must have
    either $v,w\in S$ and $\deg_G(x)=2$ (case 4), or $v\in B$,
    $w\in S$ and $x\in S_1$ (case 2).
  \item If $P_e$ has length three, then at least one of $v,w$ must be
    in $B$ and its neighbor in $P_e$ must be in $S_1$. If both $v$ and
    $w$ lie in $B$, then we are in case 5; otherwise, we have
    $v\neq w$ and we are in case 3.
  \item Finally, if $P_e$ has length four, then we have $v\neq w$ since $G$
    is $C_4$-free. Moreover, they both have to be big and their
    neighbors in $P_e$ (say $x_e,x'_e$) lie in $S_1$. The other vertex
    $y_e$ of $P_e$ must have degree two, so we are in case 6.
  \end{itemize}
  Observe in particular that if $v$ is small, then cases 5 and 6
  cannot occur. Moreover, if $v$ and $w$ are big, then only cases 1,
  5, and 6 can occur. Finally, every loop of $G'$ has type 5.
\end{proof}

In what follows, when referring to an edge $e$ with type $i$, we use
$x_e$, $x'_e$, and $y_e$ as defined in the corresponding part of
Lemma~\ref{lem:struct_edge}.
This lemma implies the following facts about the structure of regions
in $G$.
\begin{cor}
  \label{cor:forme_region}
  Let $R$ be a region of $G$. Then $V(R)$ is the disjoint union of
  three sets $B_1,B_2,D$ such that $B_i\subset N(b_i)$ for some
  $b_1,b_2\in B$, and $D$ is an independent set of $2$-vertices, each
  with a neighbor in each of $B_1$ and $B_2$.
\end{cor}

\begin{proof}
  Let $R$ be a region of $G$. By definition, there exists
  $b_1,b_2\in B$ on the boundary of every face of $R$ in
  $G''$. Therefore, in $G'$, the edges appearing in $R$ are either
  loops on $b_1$ or $b_2$ or edges between $b_1$ and $b_2$.

  Note that $V(R)$ is the set of all vertices of $G$ that disappear
  when we construct the edges of $R$ in $G'$.  For each $i\in \{1,2\}$,
  define $B_i$ as the set of vertices $v$ of $G$ such that $vb_i$ is
  contracted when constructing an edge of $R$ in $G'$. We also define
  $D$ as the set of vertices in $G$ that are suppressed when
  constructing an edge of $R$ in $G'$. By definition, we have
  $B_i\subset N(b_i)$

  By Lemma~\ref{lem:struct_edge}, since $b_1,b_2\in B$, each edge $e$
  between $b_1$ and $b_2$ in $G'$ has type 1, 5, or 6, and each loop
  around $b_1,b_2$ has type $5$. This ensures that
  $V(R)=B_1\cup B_2\cup D$ and that $D$ contains only vertices of
  degree $2$ in $G$. Using again Lemma~\ref{lem:struct_edge}, this
  implies that $D$ is an independent set.

  It remains to show that these sets are pairwise disjoint. Assume
  that there is $x\in B_1\cap B_2$. Now $xb_1$ and $xb_2$ are both
  contracted when constructing $G'$. This requires that $x\in
  S_1$. Since $b_1$ and $b_2$ are both big, we must have $b_1=b_2$, a
  contradiction.
  Further, since $b_1\in B$, no neighbor of $b_1$ is suppressed
  during the construction of $G'$. Since $B_1\subset N(b_1)$, we
  thus have $D\cap B_1=\varnothing$. By symmetry, we also have
  $D\cap B_2=\varnothing$.
\end{proof}

In the following, given a region $R$, we use the notation of
Corollary~\ref{cor:forme_region}.

\subsection{Proof of Proposition~\ref{prop:bigreg}: $G$ has Large Regions}
\label{subsec:bigreg}

Our goal in this subsection is to find a large region in $G$. To this
end, we look for a large set of consecutive faces of length $2$ in
$G'$. We first recall a result from~\cite{BCP} (Lemma~3.6 in that paper)
allowing us to find a vertex in $G'$ with few neighbors in $G'''$.

\begin{lem}[\cite{BCP}]
  \label{lem:1040}
  There exists $b_1\in B$ such that $\deg_{G'''}(b_1)\leq 40$ and
  $\deg_{G'''[B]}(b_1) \leq 10$. 
\end{lem}

\noindent
We note that the general context of~\cite{BCP} is planar graphs with girth at
least 5.  However, the proof of Lemma~\ref{lem:1040} uses only that $G$ has no
4-cycles.

Our goal is to apply a pigeonhole-like argument to find a large number
of consecutive edges between two vertices in $G''$. To this end, we
first need to control the degrees of vertices in $G''$. We begin with a
definition. The \Emph{half-edges} of $G'$ are the elements of the multiset
of pairs $(u,e)$ where $e$ is an edge incident to $u$. Note that when $e$
is a loop around $u$, there are still two half-edges $(u,e)$.
Observe also that since we fixed a plane embedding of $G$, there is a
natural cyclic ordering of the half-edges around each fixed vertex
$u$.

\begin{lem}
  \label{lem:adjacentloops}
  If $e$ is a loop around a vertex $v$ in $G'$, then one of the
  half-loops induced by $e$ must be followed or preceded by a
  half-edge $(v,vw)$ with $v\neq w$.
\end{lem}

\begin{proof}
  By Lemma~\ref{lem:struct_edge}, every loop has type $5$.  So let
  $x_e$ and $x_e'$ denote the vertices in $G$ that merged into $v$ to
  form $e$ in $G'$.  By Lemma~\ref{lem:notriangle}, either
  $\deg(x_e)>2$ or $\deg(x'_e)>2$; by symmetry, assume
  $d(x_e)>2$.  Among all neighbors of $x_e$ in $G$, other than
  $x'_e$ and $v$, choose $w$ to be one that immediately precedes or
  follows $x'_e$.  
  
  If $w$ is not suppressed in $G'$, then the half-edge $(v,vw)$
  precedes or follows $(v,e)$ or $(v,e')$. Note that $vw\notin E(G)$
  since otherwise $vwx_ex'_e$ is a 4-cycle in $G$. Thus we have
  $v\neq w$ in $G'$ and the lemma is true.
  So assume that $w$ is suppressed.
  Now $w$ has degree $2$ in $G$. Let $x$ be the neighbor of $w$
  other than $x_e$. Since $x_e$ is small, Lemma~\ref{lem:first_red}
  ensures that $x$ has degree at least 3 in $G$; hence, it is not
  suppressed in $G'$. Therefore, the half-edge $(v,vx)$ precedes or
  follows $(v,e)$ or $(v,e')$. Again, $vx\notin E(G)$ since otherwise
  $vxwx_e$ is a 4-cycle in $G$. Thus $x\neq v$ in $G'$ and the lemma is
  true.
\end{proof}

Lemma~\ref{lem:adjacentloops} implies the following relationship between
degrees of vertices in $G''$ and in $G'$.
\begin{cor}
  \label{cor:degrees}
  Every $v\in V(G')$ satisfies
  $\deg_{G''}(v)\geq \frac{\deg_{G'}(v)}{5}$.
\end{cor}

\begin{proof}
  Suppose $v\in V(G')$ and consider the half-edges around $v$ in $G'$. By
  definition, there are $\deg_{G'}(v)$ half-edges around $v$ and
  $\deg_{G''}(v)$ of them are not half-loops. So it suffices to
  prove that the number of half-loops around $v$ is at most four times the
  number of the other half-edges, i.e., at most $4\deg_{G''}(v)$. 

  Suppose $w\in N_{G'}(v)$. Consider the two half-edges $(v,e)$ and
  $(v,f)$ such that $(v,e)$, $(v,vw)$ and $(v,f)$ are consecutive
  around $v$. Let $F(w)$ be the maximum subset of $\{(v,e),(v,f)\}$
  containing only half-loops.
  Lemma~\ref{lem:adjacentloops} ensures that, for every loop, one of its 
  half-loops appears in $F(w)$ for some $w\in N_{G'}(v)$. Therefore, the
  number of half-loops around $v$ is at most
  \[2\left|\cup_{w\in N_{G'}(v)} F(w)\right|\leq 4|N_{G'}(v)|=4\deg_{G''}(v).\]
  This concludes the proof, since
  \[\deg_{G'}(v)\leq\deg_{G''}(v)+4\deg_{G''}(v)=5\deg_{G''}(v).\]
\aftermath
\end{proof}

Consider the vertex $b_1$ obtained by Lemma~\ref{lem:1040}. By
Corollary~\ref{cor:degrees}, we have
\[\deg_{G''}(b_1)\geq\frac{\deg_{G'}(b_1)}{5}\geq
\frac{\deg_G(b_1)}{5}\geq\frac{\sqrt{k}}{5}.\]
Using a pigeonhole argument, we will see that $b_1$ has some neighbor $b_2$
such that at least
$\frac{\sqrt{k}}{5\times 40}$ consecutive edges incident to $b_1$ end at
$b_2$. Note that Proposition~\ref{prop:bigreg} almost follows from
this result (with $\frac{\sqrt{k}}{50}$ replaced by $\frac{\sqrt{k}}{200}$). We
only need to refine this argument to show how to force $b_2\in B$, i.e.,
$b_2\notin S'$, where $S'=V(G')\setminus B$. To this end, we show that small
vertices are incident to few consecutive edges in $G''$.

\begin{lem}
  \label{lem:smallfaces}
  If $v\in B$ and $w\in S'$, then $(v,w)$ is on the boundary of at
  most 8 consecutive faces of length 2\ in $G''$.
\end{lem}

\begin{proof}
  Pick $v\in B$ such that there is an edge $vw\in E(G')$, with $w\in S'$.
  We consider each possible type of edge in $G'$ between $v$ and $w$. The
  type $3$ edges are a special case, which we postpone to the end.
  Since $G$ is simple, at most one edge $vw$ of $G'$ has type $1$.
  Similarly, if $G'$ has two edges $e_1$ and $e_2$ of type $2$, then
    $x_{e_1}\neq x_{e_2}$. Thus $vx_{e_1}wx_{e_2}$ is a 4-cycle in $G$, a
    contradiction. So $G'$ has at most one edge of type 2.
    Since $v\in B$ and $w\in S'$, $G'$ has no edge of type 4, 5, or 6.

  Only type 3 edges remain. We assume such an edge exists, since
  otherwise the lemma holds.  Note that $G'$ has
  no edge of type 4 (since $v\in B$), nor of type 1 (since $G$ has no
  4-cycle), nor of type 5 or 6 (since $w\in S'$). So $G'$ has at most one
  edge $f$ not of type $3$, and $f$, if it exists, has type $2$. Thus, edge $f$
  separates two blocks of consecutive type 3 edges. To prove the lemma, it 
  suffices to prove that each such block has size at most four.

  Assume that $e_1,\ldots,e_5$ are edges of type $3$ that are consecutive in
$G''$. We
  now prove that the hypotheses of Lemma~\ref{lem:smallfacesaux} are
  satisfied by the subgraph of $G$ induced by the vertices inside the
  cycle $vx_{e_1}y_{e_1}wy_{e_5}x_{e_5}$. Since each edge $e_i$ has
  type $3$, the first hypothesis holds. 

  To prove the second hypothesis holds,
  assume that some vertex $x$ is not adjacent to $v$, but $x$ lies inside some
  cycle $C=vx_{e_i}y_{e_i}wy_{e_{i+1}}x_{e_{i+1}}$. Note that $x$ is not a
  neighbor of $y_{e_i}$ or $y_{e_{i+1}}$,
  since they both have degree $2$; nor of $w$ since $e_i$ and $e_{i+1}$
  are consecutive edges in $G''$.  
  Note that $e_i$ and $e_{i+1}$ bound a face of length $2$ in $G''$
  so every vertex inside the cycle $C$ disappears when we construct
  $G'$. Thus, all these vertices are small, and either lie in $S_1$ or
  lie in $S\setminus N[B]$ and have degree $2$ in $G$. 
  Hence, $v$ is the only big
  vertex inside or on $C$ and $xv\notin E(G)$; so $x\notin \cup_{i\ge 1}S_i$.

  Since $x\notin S_1$, $x$ has degree $2$ and its two neighbors, say $y$
  and $z$, lie in $S$. Applying Lemma~\ref{lem:first_red} to edges $xy$
  and $xz$, we get that $y,z\in N[B]$. This implies that both $y$ and
  $z$ are neighbors of $v$, so $xyvz$ is a 4-cycle in $G$, a
  contradiction. Therefore, no such $x$ exists.

  Now Lemma~\ref{lem:smallfacesaux} yields a contradiction, since
  $G$ cannot contain this configuration.
\end{proof}

We can now finish the proof of Proposition~\ref{prop:bigreg}.

\begin{proof}[Proof of Proposition~\ref{prop:bigreg}]
  Let $b_1$ be a vertex in $G'''$ guaranteed by Lemma~\ref{lem:1040}. For
  each small neighbor $v$ of $b_1$ in $G'''$ and edge $vb_1$,
  Lemma~\ref{lem:smallfaces} ensures that in $G''$ edge $vb_1$ corresponds to
  at most 9 edges between $b_1$ and $v$.
  Since $\deg_{G'''}(b_1)\leq 40$, the number of such edges is at most $9\times
40=360$.
  However, by Corollary~\ref{cor:degrees}, we have
  $\deg_{G''}(b_1)\geq \frac{\deg_G(b_1)}{5}\geq
  \frac{\sqrt{k}}{5}$. Thus, there must exist a big neighbor $b_2$ of
  $b_1$ in $G''$ such that there are at least
  $\frac{\frac{\sqrt{k}}{5}-360}{\deg_{G'''[B]}(b_1)}\geq
  \frac{\sqrt{k}}{50}-36$ consecutive edges $b_1b_2$ in $G''$. By
  definition, these edges form a region of size
  $\frac{\sqrt{k}}{50}-37$ in $G$.
\end{proof}

\subsection{Proof of Proposition~\ref{prop:smallreg}: Large Regions
  are Reducible}
\label{subsec:smallreg}

In this section, we show that $G$ cannot contain arbitrarily large
regions, i.e., for $r$ large enough every $r$-region is reducible.
Note that the square of such $r$-regions consists of two cliques, with some
edges between them.   Following the terminology of
Corollary~\ref{cor:forme_region}, we denote the vertices of these cliques by
$B_1$ and $B_2$.  As before, $D$ denotes a set of independent 2-vertices, each
with one neighbor in $B_1$ and one neighbor in $B_2$.  We begin by proving that
there are only few edges between $B_1$ and $B_2$.

\begin{lem}
  \label{lem:fewedges}
  Let $R$ be an $r$-region of $G$. Each $w\in B_1\cup B_2$ has at most
  one neighbor in $B_1$, at most one in $B_2$, and at most eight in
  $D$.
\end{lem}

\begin{proof}
  Suppose $w\in B_1\cup B_2$. If $w$ has two neighbors $x$ and $y$ in
  $B_i$, then $b_ixwy$ is a 4-cycle in $G$, a contradiction. So we
  assume $w$ has at most one neighbor in each of $B_1$ and $B_2$.  In
  what follows, we assume by symmetry that $w\in B_1$.

  Suppose that $w$ has 5 consecutive neighbors $x_1,\ldots,x_5$, all
  in $D$, and denote by $y_i$ the common neighbor of $x_i$ and
  $b_2$. By Lemma~\ref{lem:smallfacesaux}, there is a vertex $z$
  inside some cycle $wx_iy_ib_2y_{i+1}x_{i+1}$ that is not adjacent to
  $b_2$.  Since $R$ is an $r$-region, $z$ disappears when we construct
  $G'$.  Since $z\notin N_G(b_2)$, vertex $z$ must be a 2-vertex.  By
  Lemma~\ref{lem:first_red}, each neighbor of $z$ is adjacent to
  $b_2$.  So $G$ contains a 4-cycle, a contradiction.  Thus, $w$ has
  at most 4 consecutive neighbors in $D$.

  Consider an edge $wx$ between these blocks of consecutive neighbors
  in $D$ where $x\in V(R)\setminus D$. Then $x$ cannot lie in $B_1$,
  otherwise $b_1wx$ is a triangle not containing $b_2$ nor any vertex
  in $B_2$. By planarity, there cannot be vertices of $D$ inside and
  outside of this triangle. Therefore $x\in B_2$.

  Since $G$ has no 4-cycle, at most one such neighbor $x$ exists, so
  $w$ has at most two blocks of consecutive neighbors in $D$. This
  proves the final assertion.
\end{proof}

We note that we can prove Theorem~\ref{thm:main} more simply (and with a better
bound on $\Delta$) if we only want the result for list-coloring.  In fact, we do
this in Section~\ref{sec5}, where we prove it for correspondence coloring. 
However, to prove this same bound for Alon--Tarsi number, as we do below, seems
to require using the Kernel Lemma (Lemma~\ref{lem:digraph}).

Proving that $G$ does not contain large regions amounts to proving
that $r$-regions of $G$ are square $L'$-colorable for a suitable
assignment $L'$. To prove this new assertion, we use an auxiliary
result about choosability, due to Bondy, Boppana, and Siegel
(see Remark~2.4\ in~\cite{AlonT92}). 
This result applies to kernel perfect
digraphs. We briefly recall the definition here. A \Emph{kernel} $K$ in a
digraph $D$ is a subset of $V(D)$ such that every vertex $v$ of $D$
satisfies: $v\in K$ if and only if $N^+(v)\cap K=\varnothing$. A
digraph is \Emph{kernel perfect} if each of its induced subgraphs has a
kernel.

\begin{lem}
  \label{lem:digraph}
  Let $D$ be a kernel perfect digraph $D$ with underlying graph
  $H$. If $L$ is a list assignment for $V(H)$ such that for all
  $v\in V(H)$, $|L(v)|\geq d^+(v)+1$, then $H$ is
  $L$-colorable.
\end{lem}

We use this lemma to reduce the problem of square $L$-coloring an
$r$-region to finding a kernel perfect orientation. We apply this
method to prove the following generic result about choosability of graphs
covered by two cliques with few edges between them.
(Our next lemma is analogous to Lemma 3.13 in~\cite{BCP}.  One major reason that our
bounds below on $|B_i|$ and $|T_i|$ are so much larger is that here we apply the lemma
to a graph $G$ that can contain triangles.)

\begin{lem}
  \label{lem:const}
  Let $H$ be a graph covered by two disjoint cliques, $B_1$ and
  $B_2$. Let $L$ be a list assignment for $V(H)$ and suppose
  $T_i\subset B_i$ for each $i\in\{1,2\}$.  Now $H$ is $L$-colorable
  if the following five conditions hold.
  \begin{enumerate}
  \item $|B_1|\geq 52811$ and $|B_2|\geq 52811$.
  \item $|T_1|\leq 4400$ and $|T_2|\leq 4400$.
  \item For each $v\in B_i$, $|N(v)\cap B_{3-i}|\leq 11$.
  \item For each $v\in T_i$, $|L(v)|\geq |B_i|-44$.
  \item For each $v\in B_i\setminus T_i$, $|L(v)|\geq |B_i|$.
  \end{enumerate}
\end{lem}

\begin{proof}
  To prove this result we construct an orientation $D$ of $H$ such
  that $D$ satisfies the hypotheses of Lemma~\ref{lem:digraph}. We
  first show that we can order the vertices $x_1,\ldots,x_{|B_1|}$ and
  $y_1,\ldots,y_{|B_2|}$ of $B_1$ and $B_2$ such that
  $T_1=\{x_1,\ldots,x_{|T_1|}\}$, $T_2=\{y_1,\ldots,y_{|T_2|}\}$ and
  every path beginning and ending in
  $\{x_{|B_1|-10},\ldots,x_{|B_1|},y_{|B_2|-10},\ldots,y_{|B_2|}\}$
  that alternates between $B_1$ and $B_2$ has length at least
  $4$. Note that a single edge may be an alternating path, so we
  require that no edge joins $x_i$ and $y_j$ whenever
  $i\geq |B_1|-10$ and $j\geq |B_2|-10$.
  
  \subsubsection*{Definition of the Orderings} 

  We now construct the vertex orderings in the previous paragraph.
  Their only non-trivial property is the absence of short alternating paths
  between the final 11 vertices in $B_1$ and those in $B_2$.
  So, our goal is to construct $Z_1\subset B_1$ and $Z_2\subset B_2$
  with $|Z_1|=|Z_2|=11$ such that no alternating path of
  length at most $3$ begins in $Z_1$ and ends in $Z_2$. To this
  end, we first define $Z_2$, then count the number of vertices in
  $B_1$ reachable from $Z_2$ with such an alternating path.

  If there exists $v\in B_1\setminus N(T_2)$\aside{$Z_2, v$} with 11 neighbors in
  $B_2$, then we take $Z_2=N_H(v)\cap B_2$. If no such vertex exists,
  then we swap the roles of $B_1$ and $B_2$, take $Z_2$ as any subset
  of $B_2\setminus (T_2\cup N(T_1))$ of size $11$ (this is always
  possible since $|B_2|\geqslant 52811\geqslant |T_2|+11|T_1|+11$),
  and let $v$ be any vertex of $B_1$.
  Since every element of $Z_2$ has at most $10$ neighbors in
  $B_1\setminus\{v\}$, we have
  $|N_{B_1}(Z_2)\setminus \{v\}|\leq 11\times 10=110$.
  Moreover, each vertex in $N_{B_1}(Z_2)\setminus\{v\}$ has at most
  $11$ neighbors in $B_1$ (one of them being in $Z_2$). Since the only
 neighbors of $v$ in $B_2$ are in $Z_2$, we obtain
  \[|N_{B_2}(N_{B_1}(Z_2))\setminus Z_2|\leq 11\times 10^2=1100.\]
  By the same argument, the set of vertices of
  $B_1$ reachable from $Z_2$ with an alternating path of length
  exactly $3$ has size
  \[|N_{B_1}(N_{B_2}(N_{B_1}(Z_2))\setminus Z_2)|\leq 1100\times
    10=11000.\]
  So the number of vertices of $B_1$ that are excluded from appearing in $Z_1$,
because of paths to $Z_2$, is at most
  \[|N_{B_1}(N_{B_2}(N_{B_1}(Z_2))\setminus Z)|+|N_{B_1}(Z_2)\setminus
    \{v\}|+|\{v\}|=11000+110+1=11111.\] 
  Further, we must also remove vertices of $T_1$.  Thus, we can choose $Z_1$ as
  desired, since $|B_1|-|T_1|-11111\geq 11$. 

  \subsubsection*{Definition of the Orientation}

  For each edge with both endpoints in the same clique, direct it toward
  the vertex of lower index. For every other edge, direct it in
  both directions, unless one of its endpoints is among the last
  11 vertices of $B_1$ or $B_2$. In this case, direct the edge toward
  this endpoint.
  
  \subsubsection*{The Orientation is Kernel-perfect}

  Let $A\subseteq V(H)$\aside{$A$, $x_p$, $y_q$}, with
  $A\neq\varnothing$. We look for a kernel of $A$. Let $x_p$
  (resp.~$y_q$) denote the vertex with smallest index in $A\cap B_1$
  (resp.~$A\cap B_2$), if it exists. If $A\cap B_1=\varnothing$, then
  $\{y_q\}$ is a kernel. Similarly, if $A\cap B_2=\varnothing$, then
  $\{x_p\}$ is a kernel. So we assume that both $x_p$ and $y_q$ are
  well-defined. We can also assume that $x_py_q\in E(H)$, since
  otherwise $\{x_p,y_q\}$ is a kernel.

  Let $x_r$ (resp.~$y_s$)\aside{$x_r$, $y_s$} denote the vertex with smallest
  index in $A\cap B_1$ (resp.~$A\cap B_2$) that is not a neighbor of $y_q$
  (resp.~$x_p$).

    We now prove that at least one of $\{x_p\}$, $\{x_p,y_s\}$,
  $\{y_q\}$ and $\{x_r,y_q\}$ is a kernel. Assume the contrary. Since
  $\{x_p,y_s\}$ is not a kernel, there exists $y_j$ such that
  $q\leq j<s$ and either there is no edge $x_py_j$ or it is directed
  only towards $y_j$. Due to the choice of $s$, this edge is present
  in $H$ and is thus directed only one way. (If $y_s$ is not well-defined, i.e.
  if $x_p$ is adjacent to every vertex in $A\cap B_2$,
  we can obtain the same result using that $\{x_p\}$ is not a kernel.)

  Similarly, using that $\{x_r,y_q\}$ is not a kernel (or only
  $\{y_q\}$ if $y_q$ is adjacent to every vertex in $A\cap B_1$), we
  have an edge $x_iy_q$ directed only towards $x_i$.

  Since $x_iy_q$ and $x_py_j$ are directed towards $x_i$ and $y_j$,
  this ensures that $x_i$ and $y_j$ are both among the final 11
  vertices of $B_1$ and $B_2$. However, this is impossible, since
  $x_iy_qx_py_j$ would be a path of length $3$ that alternates between
  $B_1$ and $B_2$ and begin and ends in the final 11 vertices of $B_1$
  and $B_2$. Thus, either $\{x_p,y_s\}$, $\{x_p\}$, $\{x_r,y_q\}$ or
  $\{y_q\}$ is a kernel of $A$. So the orientation is kernel-perfect.

  \subsubsection*{The Orientation has Small Out-degrees}

  We now prove that $|L(v)|\geq d^+(v)+1$ for every $v\in
  V(H)$. By symmetry, it suffices to prove this
  for all $v\in B_1$, i.e., $v=x_i$ whenever $i\in\{1,\ldots,|B_1|\}$.
  If $i\leq |T_1|$, i.e., $v\in T_1$, then $v$ has at most
    $|T_1|-1\leq 4399$ out-neighbors in $B_1$ and at most $11$
    out-neighbors in $B_2$. So $\deg^+(v)+1\leq 4411 \leq |B_1|-44\leq |L(v)|$.
If $|T_1|<i\leq |B_1|-11$, then $v$ has at most
    $|B_1|-12$ out-neighbors in $B_1$ and at most $11$ in $B_2$. So
    $\deg^+(v)+1\leq |B_1|\leq |L(v)|$.
If $i>|B_1|-11$, then every out-neighbor of $v$ is in $B_1$,
    so $\deg^+(v)+1\leq |B_1|\leq |L(v)|$.
\end{proof}

We now use this lemma to prove Proposition~\ref{prop:smallreg}, i.e., that
large regions are reducible for square choosability.

\begin{proof}[Proof of Proposition~\ref{prop:smallreg}]
  We use proof by contradiction. Assume that $G$ has an $r$-region $R$
  with $r\geq 475353$. Let $v_1$ and $v_2$ be adjacent vertices of $R$
  such that any vertex at distance 2 in $G$ from $\{v_1,v_2\}$ lies in
  $\{b_1,b_2\}\cup V(R)\cup N(b_1)\cup N(b_2)$. To see that such
  vertices exist, pick $v_1\in B_1$ such that each face containing
  $v_1$ is in $R$, and let $v_2$ be a neighbor of $v_1$ in
  $B_2\cup D$.

  Let \Emph{$T$} denote the set of vertices in $B_1\cup B_2$ that
  appear on a face of $G$ not in $R$.  Note that $|T|\leq 4$; this is
  because each vertex of $T$ must lie on the first or last edge of the
  $r$-region in $G'$, and each of these edges has exactly one vertex
  in each of $B_1$ and $B_2$. Let
  $T^{(1)}=N(T)\cap V(R)$\aside{$T^{(1)}$, $T^{(2)}$ $T^{(3)}$},
  $T^{(2)}=N(T^{(1)})\cap V(R)$ and $T^{(3)}=N(T^{(2)})\cap V(R)$, so
  that for $1\leqslant i\leqslant 3$, $T\cup \cdots \cup T^{(i)}$ is
  the set of vertices of $V(R)$ at distance at most $i$ from $T$ (with
  the distance taken in $V(R)$). By Lemma~\ref{lem:fewedges}, each
  vertex of $T$ has at most 10 neighbors in $V(R)$, so
  $|T^{(1)}|\leq 40$, $|T^{(2)}|\leq 400$ and $|T^{(3)}|\leq 4000$.

  By minimality, $(G-v_1v_2)^2$ has an $L$-coloring
  $\vph$.\aside{$\vph$, $B'_i$} Let $B'_i=B_i\setminus N[T]$.  We
  uncolor the vertices of $B'_1\cup B'_2\cup D$. We also define $T_i$
  as the set of vertices of $B'_i$ with some colored neighbor from
  $V(R)$ in $G^2$, i.e., $T_i=B'_i\cap (T^{(2)}\cup T^{(3)})$.
  Finally, let $H=G^2[B'_1\cup B'_2]$.\aside{$T_i$, $H$} Note that
  $B'_1$ and $B'_2$ are cliques in $H$. Moreover, they are disjoint
  since $B'_1\cap B'_2\subset B_1\cap
  B_2=\varnothing$. 

  Our goal is now to apply Lemma~\ref{lem:const} to $L'$-color $H$,
  where $L'$ is the list assignment formed from $L$ by removing all
  colors already used on vertices at distance at most $2$:
  \[L'(v)\aside{$L'(v)$}=L(v)\setminus\{\vph(w), w\in N^2(v)\setminus (V(H)\cup
    D)\}.\] 

  We prove that the hypotheses of Lemma~\ref{lem:const} are satisfied.

  Suppose $v\in B'_1$. Now
    $|N^2(v)\cap B'_2|= |N(v)\cap B'_2|+\sum_{w\in N(v)}
    |N(w)\cap B'_2|$.  By Lemma~\ref{lem:fewedges}, for each
    $w\in V(R)$, $|N(w)\cap B'_2|\leq 1$. Moreover, if
    $w\notin V(R)$, then $|N(w)\cap B'_2|=0$, unless $w=b_2$. Since
    $b_2\notin N(v)$, we get
    \[|N^2(v)\cap B'_2|\leq 1+|N(v)\cap V(R)|\leq 11.\]

    Suppose $v\in B'_1\setminus T_1$. By definition, $v$ is distance
    at least four from $T$ (in $V(R)$), hence at distance at least
    three (in $V(R)$) from $N[T]$, the set of colored vertices of
    $V(R)$. So the only colored neighbors of $v$ in $G^2$ are in
    $\{b_1,b_2\}\cup (N(b_1)\setminus B'_1)$. Hence, we have
    \[|L'(v)|\geq k+2-(2+k-|B'_1|)= |B'_1|.\]

    Suppose $v\in T_1$. By construction, its colored neighbors in
    $G^2$ are in
    $\{b_1,b_2\}\cup (N(b_1)\setminus B'_1)\cup T\cup T^{(1)}$. Since
    $|T|+|T^{(1)}|\leq 44$, we have $|L'(v)|\geq |B'_1|-44$.

    We already saw that
    $|T_1|\leq |T^{(2)}\cup T^{(3)}|\leq 400+4000=4400$.  There are
    $r+1$ edges in the region $R$ (in $G'$). Every such edge (except
    $b_1b_2$ if it exists) corresponds to a path containing a vertex
    in $B_1$. By Lemma~\ref{lem:fewedges}, each vertex in $B_1$
    accounts for at most nine of them. Therefore,
    $|B_1|\geq \frac{r}{9}$.
    Observe also that $|N[T]\cap B_1|\leq 6$ since
    $|T\cap B_1|=2$ and, by Lemma~\ref{lem:fewedges}, every vertex of
    $B_1\cup B_2$ has at most one neighbor in each of $B_1$ and $B_2$. We
    thus obtain:
    \[|B'_1|\geq |B_1|-|N[T]\cap B_1|\geq \frac{r}{9}-6\geq 52811.\]
  We can thus apply Lemma~\ref{lem:const} to find an $L'$-coloring of $H$. 

  It remains to color the vertices in $D$. Note that each has at most
  $2\sqrt{k}$ neighbors and $k+2$ colors. 
  So we can greedily color the vertices in $D$.
\end{proof}

This completes the proof of Theorem~\ref{thm:main}.

\subsection{Extension to correspondence coloring}
\label{sec5}
In this section, we prove the following extension of
Theorem~\ref{thm:main} to correspondence coloring.
(Recall the definition of correspondence coloring from the end of
Section~\ref{sec:reducibles}.)

\begin{thm}
  \label{thm:main_cc}
  There exists $\Delta_0$ such that if $G$ is a plane graph with no
  4-cycles and with $\Delta(G)\geqslant \Delta_0$, then
  $\chi_{corr}(G^2)\leqslant \Delta+2$.
\end{thm}

Let $\Delta_0=2642900^2=6984920410000$, and fix
$k\geqslant \Delta_0$.\aside{$\Delta_0$, $k$} We prove
Theorem~\ref{thm:main_cc} by contradiction. Suppose the theorem is
false; let $G$ be a counterexample minimizing $|V(G)|+|E(G)|$, and let
$C$ be a $(k+2)$-correspondence assignment for $G^2$ such that $G^2$
has no $C$-coloring. So $C$ assigns, to each pair of vertices $(v,w)$
adjacent in $G^2$, a partial matching $C_{vw}$ between
$\{v\}\times \{1,\ldots,k+2\}$ and $\{w\}\times \{1,\ldots,k+2\}$.

We claim that Lemmas~\ref{sec3:lem1} through~\ref{lem:fewedges} still hold for
$G$ in
this new setting, since in proving each lemma we color vertices using only
that they have more available colors than colored neighbors. So
Proposition~\ref{prop:bigreg} also still holds. It thus suffices to prove
the following generalization of Proposition~\ref{prop:smallreg} for $G$.

\begin{prop}
  \label{prop:smallreg++}
  Every $r$-region of $G$ satisfies $r\leqslant 52821$.
\end{prop}

Assuming this proposition holds, we can conclude. Indeed,
Propositions~\ref{prop:bigreg} and~\ref{prop:smallreg++} imply that
$\frac{\sqrt{k}}{50}-37< 52821$, i.e., that
$k< 2642900^2=6984920410000=\Delta_0$, a contradiction.

It thus remains to prove that large regions are reducible, by
generalizing Lemma~\ref{lem:const}. The argument using kernel-perfect
orientations is no longer valid, since Lemma~\ref{lem:digraph} does
not extend to correspondence coloring.

\begin{lem}
  \label{lem:twocliques}
  Let $H$ be a graph covered by two disjoint cliques, $B_1$ and $B_2$, each of
  size $n$.
  Suppose there exist $T_1\subset B_1$ and $T_2\subset B_2$, and a function $f$
  satisfying the four properties below.  If $n\ge 5863$,
  then every $f$-correspondence assignment $C$ admits a $C$-coloring.
  \begin{enumerate}
  \item For each $v\in (B_1\setminus T_1)\cup(B_2\setminus T_2)$,
    we have $f(v)\geqslant n$.
  \item For each $v\in T_1\cup T_2$, we have $f(v)\geqslant n-44$.
  \item $|T_1|\leqslant 4400$ and $|T_2|\leqslant 4400$.
  \item $\Delta(H)-n+1\leqslant 11$.
  \end{enumerate}
\end{lem}

\begin{proof}
Let \Emph{$A$} be a subset of $B_1\setminus T_1$ with $|A|=\Delta(H)+1-n$.
Since each vertex $v\in (B_1\setminus T_1)\cup (B_2\setminus T_2)$ has
$f(v)\ge n$ and $\Delta(H)-|A|=n-1$, it is easy to greedily $C$-color all
vertices
of $H-A$.  For example, greedily color all vertices of $T_2$, followed by those
of $B_2\setminus T_2$, followed by those of $T_1$, followed by those of
$B_1\setminus(T_1\cup A)$.  This greedy coloring is possible because at the
time we color each vertex it has more available colors than colored neighbors. 

We generally follow this approach.  However, we modify it so that
after we color $H-A$ each vertex in $A$ still has $|A|$ available
colors, and we can extend the coloring to $A$.  To do this, for each
vertex $v\in A$\aside{$v$} we will repeatedly ``save a color'', before
greedily coloring the other vertices.  To accomplish this we pick
vertices $w\in N(v)\cap B_2$\aside{$w$} and
$x\in B_1\setminus N(w)$\aside{$x$, $\alpha$, $\beta$}.  Now we color
$w$ and $x$ with some colors $\alpha$ and $\beta$ (possibly with
$\alpha=\beta$) such that $\alpha$ and $\beta$ forbid the same color
on $v$.  For each $v\in A$, we must save a color $|N(v)\cap B_2|$
times.  After doing so, we color the remaining vertices greedily (as
in the previous paragraph), ending with the vertices of $A$.  The only
change is that we must ensure that each of the final 11 vertices we
color in $B_2$ has no colored neighbor in $B_1$.  In the process of
saving colors for vertices in $A$, we color at most $11^2$ vertices in
$B_1$.  Each of these forbids at most 11 vertices in $B_2$ from
appearing among the final 11\ in $B_2$, for a total of at most $11^3$
vertices in $B_2$ forbidden.  Similarly, we color at most $11^2$
vertices in $B_2$, and these are obviously forbidden from appearing
among the final 11 vertices in $B_2$.  Thus, we can choose the desired
11 final vertices in $B_2$ (after saving colors for the vertices in
$A$), since $|B_2|\ge |T_2|+ 11^3+11^2+11$.

Note that, while
saving colors for some vertex $v\in A$, we color all neighbors of $v$ in
$B_2$.  As a result, we need that no two vertices in $A$ have a common
neighbor in $B_2$.  Each vertex $v\in A$ has at most 11 neighbors in $B_2$,
and each of these neighbors has at most 10 other neighbors in $B_1$.  Thus, each
$v\in A$ forbids at most $11(10)$ other vertices from $A$.  So, to
pick the desired $A$, we need $|B_1|> |T_1|+10(110+1)$.

Now, for each $v\in A$, we repeat the following $|N(v)\cap B_2|$
times.  Choose uncolored vertices $w\in N(v)\cap B_2$ and
$x\in B_1\setminus N(w)$.  Note that if $N(v)\subset B_1$, there is
nothing to do at all, hence we may assume that the vertex $w$
exists. Let $g(v)$, $g(w)$, and $g(x)$\aside{$g(v)$, $g(w)$ $g(x)$}
denote the number of remaining available colors for $v$, $w$, and $x$.

Without loss of generality, we assume that the bounds of Hypotheses
1. and 2. are tight, so that $f(y)=n-44$ for all $y\in T_1\cup T_2$,
and $f(y)=n$ otherwise. Since $A\cap T_1=\varnothing$, we have
$f(v)=n\geqslant f(w)$, hence we may assume that $C_{vw}$ saturates
$\{w\}\times \{1,\cdots,f(w)\}$ (otherwise, add arbitrary edges until
this is the case). Thus, each color available for $w$ forbids a color
for $v$; similarly for colors available for $x$.  By Pigeonhole, if
$g(w)+g(x)>n$, then there exist colors $\alpha$ and $\beta$, available
for $w$ and $x$ respectively, that both forbid the same color on $v$.
Suppose that thus far we have saved a total of $i$ colors for vertices
in $A$.  Therefore, the $i$ colored vertices of $B_2$ forbid $i$
colors for $w$, and its neighbors in $B_1$ forbid at most 11 colors,
so that we have $g(w)\ge f(w)-i-11\ge n-i-11\ge n-131$ and, similarly,
$g(x)\ge n-131$.  We can assume that $g(v)\le f(v)\le n$.  And clearly
$2(n-131) > n$.  Thus, the desired colors $\alpha$ and $\beta$ exist.
\end{proof}

It is worth noting that the $\Delta_0$ given by our proof of
Theorem~\ref{thm:main_cc}, namely $2642900^2$, is much smaller than
that arising from our proof of Theorem~\ref{thm:main}, namely
$23769500^2$.  By adapting the statement and proof of
Lemma~\ref{lem:twocliques}, we can extend the main result
in~\cite{BCP} to correspondence coloring (while also modestly
decreasing the $\Delta_0$ arising from that proof).

\section*{Acknowledgments}
Thanks to four referees for their feedback; in particular, one wrote a very
thorough and useful report.  Thanks also to Zden\v{e}k Dvo\v{r}\'{a}k and
Jean-S\'{e}bastien Sereni for their helpful comments after carefully reading
Section~\ref{sec3}.  Zden\v{e}k caught a few errors in an earlier version of
this paper.
\bibliographystyle{plain}
\bibliography{no4-cycles}

\end{document}